\documentclass[12pt,leqno]{article}
\usepackage{amsfonts}
\usepackage{extpfeil}
\usepackage{amsmath, dsfont, amsthm, amsfonts, amssymb, color}
\usepackage{mathrsfs}
\usepackage{color}
\usepackage{stmaryrd}
\usepackage{bbm}
\newtheorem{thm}{Theorem}[section]
\newtheorem{cor}[thm]{Corollary}
\newtheorem{lem}[thm]{Lemma}

\newtheorem{exa}[thm]{Example}
\newtheorem{rem}[thm]{Remark}
\theoremstyle{definition}

\newcommand{\scr}[1]{\mathscr #1}
\definecolor{wco}{rgb}{0.5,0.2,0.3}

\numberwithin{equation}{section} \theoremstyle{remark}

\newcommand{\ua}{\uparrow}

\title{{\bf
Exponential Ergodicity in $\W_1$ for SDEs with Distribution Dependent Noise and Partially Dissipative Drifts}\footnote{Supported in
 part by  National Key R\&D Program of China (No. 2022YFA1006000) and NNSFC (12271398,12101390).} }
\author{
{\bf  Xing Huang $^{a)}$,  Huaiqian Li $^{a)}$, Liying Mu $^{a)}$  }\\
\footnotesize{ a)Center for Applied Mathematics, Tianjin
University, Tianjin 300072, China}\\
\footnotesize{  xinghuang@tju.edu.cn,\quad huaiqian.li@tju.edu.cn,\quad liying\_mu01@tju.edu.cn }}
\begin{document}

\allowdisplaybreaks
\def\R{\mathbb R}  \def\ff{\frac} \def\ss{\sqrt} \def\B{\mathbf
B} \def\W{\mathbb W}
\def\N{\mathbb N} \def\kk{\kappa} \def\m{{\bf m}}
\def\ee{\varepsilon}\def\ddd{D^*}
\def\dd{\delta} \def\DD{\Delta} \def\vv{\varepsilon} \def\rr{\rho}
\def\<{\langle} \def\>{\rangle} \def\GG{\Gamma} \def\gg{\gamma}
  \def\nn{\nabla} \def\pp{\partial} \def\E{\mathbb E}
\def\d{\text{\rm{d}}} \def\bb{\beta} \def\aa{\alpha} \def\D{\scr D}
  \def\si{\sigma} \def\ess{\text{\rm{ess}}}
\def\beg{\begin} \def\beq{\begin{equation}}  \def\F{\scr F}
\def\Ric{\text{\rm{Ric}}} \def\Hess{\text{\rm{Hess}}}
\def\e{\text{\rm{e}}} \def\ua{\underline a} \def\OO{\Omega}  \def\oo{\omega}
 \def\tt{\tilde} \def\Ric{\text{\rm{Ric}}}
\def\cut{\text{\rm{cut}}} \def\P{\mathbb P} \def\ifn{I_n(f^{\bigotimes n})}
\def\C{\scr C}      \def\aaa{\mathbf{r}}     \def\r{r}
\def\gap{\text{\rm{gap}}} \def\prr{\pi_{{\bf m},\varrho}}  \def\r{\mathbf r}
\def\Z{\mathbb Z} \def\vrr{\varrho}
\def\L{\scr L}\def\Tt{\tt} \def\TT{\tt}\def\II{\mathbb I}
\def\i{{\rm in}}\def\Sect{{\rm Sect}}  \def\H{\mathbb H}
\def\M{\scr M}\def\Q{\mathbb Q} \def\texto{\text{o}}
\def\Rank{{\rm Rank}} \def\B{\scr B} \def\i{{\rm i}} \def\HR{\hat{\R}^d}
\def\to{\rightarrow}\def\l{\ell}\def\iint{\int}
\def\EE{\scr E}\def\Cut{{\rm Cut}}
\def\A{\scr A} \def\Lip{{\rm Lip}}
\def\BB{\scr B}\def\Ent{{\rm Ent}}\def\L{\scr L}
\def\R{\mathbb R}  \def\ff{\frac} \def\ss{\sqrt} \def\B{\mathbf
B}
\def\N{\mathbb N} \def\kk{\kappa} \def\m{{\bf m}}
\def\dd{\delta} \def\DD{\Delta} \def\vv{\varepsilon} \def\rr{\rho}
\def\<{\langle} \def\>{\rangle} \def\GG{\Gamma} \def\gg{\gamma}
  \def\nn{\nabla} \def\pp{\partial} \def\E{\mathbb E}
\def\d{\text{\rm{d}}} \def\bb{\beta} \def\aa{\alpha} \def\D{\scr D}
  \def\si{\sigma} \def\ess{\text{\rm{ess}}}
\def\beg{\begin} \def\beq{\begin{equation}}  \def\F{\scr F}
\def\Ric{\text{\rm{Ric}}} \def\Hess{\text{\rm{Hess}}}
\def\e{\text{\rm{e}}} \def\ua{\underline a} \def\OO{\Omega}  \def\oo{\omega}
 \def\tt{\tilde} \def\Ric{\text{\rm{Ric}}}
\def\cut{\text{\rm{cut}}} \def\P{\mathbb P} \def\ifn{I_n(f^{\bigotimes n})}
\def\C{\scr C}      \def\aaa{\mathbf{r}}     \def\r{r}
\def\gap{\text{\rm{gap}}} \def\prr{\pi_{{\bf m},\varrho}}  \def\r{\mathbf r}
\def\Z{\mathbb Z} \def\vrr{\varrho}
\def\L{\scr L}\def\Tt{\tt} \def\TT{\tt}\def\II{\mathbb I}
\def\i{{\rm in}}\def\Sect{{\rm Sect}}  \def\H{\mathbb H}
\def\M{\scr M}\def\Q{\mathbb Q} \def\texto{\text{o}} \def\LL{\Lambda}
\def\Rank{{\rm Rank}} \def\B{\scr B} \def\i{{\rm i}} \def\HR{\hat{\R}^d}
\def\to{\rightarrow}\def\l{\ell}
\def\8{\infty}\def\I{1}\def\U{\scr U} \def\n{{\mathbf n}}

\def\Rd{\mathbb{R}^d}
\def\E{\mathbb{E}}
\def\P{\mathscr{P}}
\def\D{\mathrm{D}}
\def\Z{\mathbb{Z}}
\def\i{\mathrm{i}}
\def\L{\mathscr{L}}
\def\M{\mathcal{M}}
\def\d{\textup{d}}
\def\D{\textup{D}}
\def\U{\mathrm{U}}
\def\ch{{\bf 1}}
\def\BE{\textup{BE}}
\def\BV{\textup{BV}}
\def\Cap{\textup{Cap}}
\def\CD{\textup{CD}}
\def\Ch{\textup{Ch}}
\def\csch{\textup{csch}}
\def\det{\textup{det}}
\def\Hess{\textup{Hess}}
\def\Lip{\textup{Lip}}
\def\MCP{\textup{MCP}}
\def\Ric{\textup{Ric}}
\def\tr{\textup{tr}}
\def\supp{\textup{supp}}
\def\RCD{\textup{RCD}}
\def\vol{\textup{vol}}
\def\<{\langle}
\def\>{\rangle}
\def\Proof.{\noindent{\bf Proof. }}
\def\ll{\lambda}
\def\aa{\alpha}
\def\gg{\gamma}
\def\dd{\delta}
\def\nn{\nabla}

\def\e{\text{\rm{e}}}
\def\F{\mathcal{F}}
\def\N{\mathbb{N}}
\def\SF{\mathscr{F}}
\def\ff{\frac}
\def\Ent{\textup{Ent}}
\def\Leb{\textup{Leb}}
\def\div{\textup{div}}
\def\Var{\textup{Var}}
\def\loc{\textup{loc}}
\def\fin{\hfill$\square$}
\def\dis{\displaystyle}
\def\newdot{{\kern.8pt\cdot\kern.8pt}}
\def\bint{\hskip2pt-\hskip-12pt\int}
\def \bt{b\left(t, X_t, \mathscr{L}(X_t)\right)}

\maketitle

\begin{abstract}

Being concerned with ergodicity of McKean--Vlasov SDEs, we establish a general result on exponential ergodicity in the $L^1$-Wasserstein distance.
The result is successfully applied to non-degenerate and multiplicative Brownian motion cases, degenerate second order systems, and even the additive $\alpha$-stable noise, where the coefficients before the noise are allowed to be distribution dependent and the drifts are only assumed to be partially dissipative. Our results considerably improve existing ones whose coefficients before the noise are distribution-free.

\end{abstract}

%{\bf MSC 2020:} primary  60D05, 58J65; secondary 60J60, 60J76


\section{Introduction}\hskip\parindent
Let $\P(\Rd)$ be the class of all probability measures on $\Rd$ equipped with the weak topology. For $p\geq1$,
define
$$\P_p(\Rd):=\left\{\mu \in \P(\Rd):\ \mu(|\cdot|^p)< \infty\right\}.$$
The $L^p$-Wasserstein distance on $\P_p(\Rd)$ is given by
$$\W_p(\mu, \nu):=\inf_{\pi \in \mathscr{C}(\mu, \nu)}\left(\int_{\Rd\times \Rd} |x-y|^p \,\pi(\d x, \d y)\right)^{\frac{1}{p}}, \quad \mu, \nu \in \P_p(\Rd),$$
where $\mathscr{C}(\mu, \nu)$ denotes the class of all couplings of $\mu$ and $\nu$.

 Consider the following McKean--Vlasov SDE or distribution dependent SDE (abrev DDSDE) %%%%%%后面也在用
 on $\Rd$:
\begin{equation}\label{GPS0}
\d X_t= b(X_t, \L_{X_t})\d t+  \sigma(X_{t}, \L_{X_t}) \d B_t,
\end{equation}
 where $\L_{X_t}$ denotes the law of $X_t$,
$$b:\Rd \times \mathscr{P}(\Rd)
 \rightarrow
   \Rd, \quad \sigma:\Rd \times \mathscr{P}(\Rd)
    \rightarrow
    \Rd \otimes \mathbb{R}^n, $$
 are measurable, and $(B_t)_{t\geq 0}$ is an $n$-dimensional standard Brownian motion on a complete filtration probability space $(\Omega, \scr F, (\scr F_t)_{t\geq 0},\mathbb{P})$. This type of SDE is proposed in \cite{Mckean 1967} and is used to characterize nonlinear Fokker--Planck equations; see e.g. \cite{Barbu 2020}.
 In this work, we are concerned with the exponential ergodicity of \eqref{GPS0} in $\W_1$.

%In the classical SDE,
For the classical SDE,
i.e., $b(x,\mu)=b(x),\sigma(x,\mu)=\sigma(x)$, various types of exponential ergodicity have been investigated. For instance,  when $\sigma=I_{d\times d}$, the exponential decay in $L^p$-Wasserstein distance for all $p\in[1,\infty)$ is obtained in \cite{LDJ 2016}. Based on the celebrated result in  \cite{Hairer 2021}, a quantitative Harris-type theorem for SDEs with $\sigma=I_{d\times d}$ is presented in \cite{Eberle 2019}. The exponential ergodicity for SDEs driven by L\'{e}vy noises is studied in \cite{LMJ 2020, Majka 1967}.
In addition, the ergodic properties of %fundamental
 solutions to kinetic Langevin equations are considered in \cite{BJH 2024,Eberle 2018}.

For the McKean--Vlasov SDE, a wide range of methods have been employed to study the long-time behavior in various senses. When $\sigma(x,\mu)=\sigma(x)$, the author in \cite{WFY 2023} adopts the Lyapunov function and the coupling by reflection to derive the exponential ergodicity in quasi-Wasserstein distance. In addition,
 \cite{WFY 2023 SPA} uses Girsanov's theorem together with Harris type theorem to establish the exponential ergodicity of singular McKean--Vlasov SDEs in total variation distance. Under uniformly dissipative conditions, \cite{RPP 2021} takes advantage of the log-Harnack inequality and the Talagrand inequality to obtain exponential convergence in both relative entropy and $L^2$-Wasserstein distance, which extends the results in \cite{Carrillo 2020, WFY 2018}. In the case of $\sigma=I_{d\times d}$, the authors in \cite{Guillin 2022} apply the uniform log-Sobolev inequality for mean field interacting particle system to derive the exponential ergodicity of  McKean--Vlasov SDEs in mean field relative entropy.

Finally, let us concentrate on exponential ergodicity of \eqref{GPS0} in $\W_1$. %Assume {\color{blue}that}
Let $\sigma=I_{d\times d}$ and $b$ be of the form
$$b(x,\mu)=b_0(x)+\int_{\R^d}Z(x,y)\mu(\d y), \quad x\in \Rd,\, \mu\in \P(\Rd),$$
where $b_0:\Rd \to \Rd$ and $Z:\Rd\times \Rd \to \Rd$ are some given measurable functions. %Moreover,
Assume that %$b$ satisfies
\begin{equation*}
\<b_0(x)-b_0(y), x-y\>\leq \kappa(|x-y|)|x-y|,\quad |Z(x,y)-Z(\bar{x},\bar{y})|\leq \kappa_0(|x-\bar{x}|+|y-\bar{y}|),
\end{equation*}
for some positive constant %$\kappa_0>0$
 $\kappa_0$
and some continuous function  $\kappa:[0, \infty)\to \R$ such that
$$\int_0^1\kappa^+(r)\d r<\infty, \quad \limsup_{r\to\infty}\kappa(r)<0,$$
where $\kappa^+:=\max\{\kappa,0\}$. %Here, $\kappa^+(r)$ denotes the positive part of the function $\kappa(r)$, i.e., $\kappa^+(r) = \max\{\kappa(r), 0\}.$
 Under these conditions, the aforementioned paper \cite{Eberle 2019} derives the exponential ergodicity of \eqref{GPS0} in $\W_1$, provided that $\kappa_0$ is small enough. Basing on \cite{Eberle 2019}, the authors in \cite{LW 2021} derive the explicit convergence rate. One can refer to \cite{Schuh 2024} for the kinetic case. \cite{WFY 2023} extends the result in \cite{Eberle 2019} to the case $\sigma=\sigma(x)$ and %derive
derives the
 exponential ergodicity in $\W_{\phi}$ for some reasonable concave function
$\phi:[0,\infty)\to [0,\infty)$, where $\W_{\phi}$ denotes the Wasserstein distance %with the metric $d(x,y)$ replaced by the concave function $\phi(|x-y|)$
defined by the cost function $\phi(|\cdot-\cdot|)$.
For the additive L\'{e}vy noise case in \cite{ LMJ 2021}, by using the method of asymptotic refined basic coupling, the authors obtain the exponential ergodicity in $\W_1$; see also \cite{LWZ 2023} for the kinetic case.

In the general case $\sigma=\sigma(x,\mu)$, under long-distance dissipative conditions, \cite{ZSQ 2023} establishes exponential decay of \eqref{GPS0} in  $\W_1$ when starting from any Dirac measure $\delta_x$. However, %as
the author of \cite{ZSQ 2023} pointed out that the result can not be extended to any initial distribution $\mu \in \P_1(\Rd)$.

In this paper, we investigate the exponential ergodicity in $\W_1$ of
\eqref{GPS0} for the general case $\sigma=\sigma(x,\mu)$ and the initial distribution is allowed to be any $\mu \in \P_1(\Rd)$, which extends the results in \cite{ZSQ 2023}. The essential novelty %includes two aspects,
 are twofold.
 %one is
On one hand,
  the coefficients before % noise can be allowed
the noise are allowed
 to be distribution dependent,  and %the other is
on the other hand,
  the coefficients are %partial
  partially
  dissipative. The trick of employing asymptotic reflection coupling %realized
  in \cite{Eberle 2019,LW 2021,Schuh 2024,WFY 2023} or asymptotic refined basic coupling in \cite{ LMJ 2021, LWZ 2023} seems unavailable. We illustrate this point in the Brownian motion case below. In fact, when $\sigma$ is distribution dependent, we need to compare two diffusion processes with different diffusion coefficients:
$$\d X_t^i=\sigma^i(X_t^i)\d B_t,\ \ i=1,2.$$
Applying the It\^{o}--Tanaka formula for $|X_t^1-X_t^2|$, a singular term
\begin{align}\label{cty}\frac{%\W_1(\L(X_t),\L(Y_t))^2
\|\sigma^1(X_t^1)-\sigma^2(X_t^2)\|_{\mathrm{HS}}^2}{|X_t^1-X_t^2|}
\end{align}
appears,   where $\|\cdot\|_{\mathrm{HS}}$ denotes the Hilbert--Schmidt norm, which leads to essential difficulty to estimate $\E|X_t^1-X_t^2|$.

The overall idea of our study to overcome the difficulty is to first fix a distribution in the coefficients and transform equation  \eqref{GPS0} into a classical SDE, thereby eliminating the influence of the distribution dependent term. Subsequently, we employ the reflection coupling to establish the exponential ergodicity of the classical SDE. In the second step, we derive the estimate for $\E|X_t^1-X_t^2|^2$ which can overcome the singularity of \eqref{cty}. Combining these results together, we then demonstrate the exponential ergodicity of the original \eqref{GPS0}. In fact, the idea has been used in \cite{WFY 2023 SPA} to derive the exponential ergodicity in total variation distance in the case that $\sigma$ is distribution-free with Brownian motion noise.

The remaining of the paper is %constructed in the following.
 structured as follows.
 A general result on the exponential ergodicity in $\W_1$ is established in Section 2. In Section 3, %the Brownian motion noise case is considered, and both the non-degenerate and kinetic case
we focus on the Brownian motion noise case and the non-degenerate and the kinetic case
  are investigated. We also derive the exponential ergodicity in $\W_1$ for McKean--Vlasov SDEs driven by the $\alpha$-stable noise in Section 4, %. and similar to Section %2 Similar
where, similar
 to Section 3,  both the non-degenerate and the kinetic cases %case
are %involved in.
studied.

\section{A general result}
Let $Z_t$ be an $n$-dimensional L\'{e}vy process on some complete filtration probability space $(\Omega, \scr F, (\scr F_t)_{t\geq 0},\mathbb{P})$. Let $b:\R^d\times\scr P(\R^d)\to\R^d$ and $\sigma:\R^d\times \P(\R^d) \to\R^d\otimes\R^{n}$ be %are
measurable and bounded on bounded set. Consider
\begin{align}\label{GPS}
\d X_t= b(X_t, \L_{X_t})\d t+  \sigma(X_{t-}, \L_{X_t}) \d Z_t.
\end{align}
Since we %investigate
 are concerned with
 the invariant probability measure and the ergodicity of \eqref{GPS}, we will also consider the time-homogeneous decoupled SDEs with parameter $\mu\in\scr P_1(\R^d)$, i.e.,
\begin{align}\label{GPS12}
\d X_t^{\mu}= b(X_t^{\mu}, \mu)\d t+  \sigma(X_{t-}^{\mu}, \mu) \d Z_t.
\end{align}
Suppose that \eqref{GPS} and \eqref{GPS12} are well-posed. Let %$(P_t^{\mu})^\ast\eta$ and
$P_t^\ast \eta$ (resp. $(P_t^{\mu})^\ast\eta$) be the distribution of the solution to %\eqref{GPS12} and
\eqref{GPS} (resp. \eqref{GPS12}) with initial distribution $\eta\in\scr P_1(\R^d)$, respectively. Since $b$ and $\sigma$ do not depend on $t$, we have the following semigroup property for $\{P_t^\ast\}_{t\geq0}$ and $\{(P_t^{\mu})^\ast\}_{t\geq0}$, i.e.,
\begin{equation}\label{semigroup property}
  P_{t+s}^{*}=P_{t}^{*}P_{s}^{*}, \quad  (P_{t+s}^{\mu})^\ast=(P_t^{\mu})^\ast(P_s^{\mu})^\ast,\quad\quad  t,s \geq0.
\end{equation}

Our main result in this section is presented in the next theorem.
\begin{thm}\label{ABC}
Assume that for any $\mu\in\scr P_1(\R^d)$, there exists a unique invariant probability measure $\Gamma(\mu)\in\scr P_1(\R^d)$ to \eqref{GPS12} such that $$\W_1((P_t^{\mu})^\ast\eta,\Gamma(\mu))\leq c_0\e^{-\lambda_0 t}\W_1(\eta,\Gamma(\mu)),\quad t\geq 0,\,\eta,\mu\in\scr P_1(\R^d)$$
for some constants $c_0\geq 1$ and $\lambda_0>0$ not depending on $\mu$.

\begin{enumerate}
\item[(i)]
Suppose %furthermore
that there exists a continuous function $G:[0,\infty)\to[0,\infty)$ satisfying
$$\inf_{t>\frac{\log c_0}{\lambda_0}}\frac{G(t)}{1-c_0\e^{-\lambda_0 t}}\in[0,1)$$
such that
$$\W_1((P_t^{\mu_1})^\ast\Gamma(\mu_2),\Gamma(\mu_2))\leq  G(t)\W_1(\mu_1,\mu_2),\quad t\geq 0,\,\mu_1,\mu_2\in\scr P_1(\R^d).$$
%then
 Then
there exists a unique invariant probability
measure $\mu^\ast\in\scr P_1(\R^d)$ %to
for
\eqref{GPS}.
\item[(ii)] %Moreover, if
Let $\nu\in\scr P_1(\R^d)$ be an invariant probability measure for \eqref{GPS}.   Assume that there exist a constant $t_1>0$ and a continuous function $H:[0,t_1]\to[0,\infty)$ with
$$H(t_1)+c_0\e^{-\lambda_0 t_1}<1$$
such that $$\W_1(P_{t}^\ast\eta,(P_{t}^{\nu})^\ast\eta)\leq H(t)\W_1(\eta,\nu),\quad \eta\in\scr P_1(\R^d),\ \ t\in[0,t_1].$$
Then there exist constants $c,\lambda>0$ such that
$$\W_1(P_t^\ast\eta,\nu)\leq c\e^{-\lambda t}\W(\eta,\nu), \quad t\geq0,\, \eta\in\P_1(\R^d).$$
\end{enumerate}
\end{thm}
\begin{proof}%[Proof of Theorem \ref{ABC}]
(i) Note that for any $\mu \in \P_1(\R^d)$ and any $ t\geq 0,$
$$\Gamma(\mu)=(P_t^{\mu})^\ast\Gamma(\mu).$$
By the triangle inequality, we have
\begin{align*}
\W_1(\Gamma(\mu_1),\Gamma(\mu_2))&\leq \W_1(\Gamma(\mu_1),(P_t^{\mu_1})^\ast\Gamma(\mu_2)) +\W_1((P_t^{\mu_1})^\ast\Gamma(\mu_2),\Gamma(\mu_2))\\
&\leq c_0%\e^{-\lambda_0 t}
\e^{-\lambda_0 t}
\W_1(\Gamma(\mu_1),\Gamma(\mu_2))+ G(t)\W_1(\mu_1,\mu_2), \quad \mu_1,\mu_2 \in\scr P_1(\R^d),\,t\geq 0.
\end{align*}
%begin{align*}
%\W_1(\Gamma(\mu),\Gamma(\bar{\mu}))&\leq \W_1(\Gamma(\mu),(P_t^{\mu})^\ast\Gamma(\bar{\mu})) +\W_1((P_t^{\mu})^\ast\Gamma(\bar{\mu}),(P_t^{\bar{\mu}})^\ast\Gamma(\bar{\mu}))\\
%&\leq c_0\e^{-\lambda_0 t}\W_1(\Gamma(\mu),\Gamma(\bar{\mu}))+\kappa G(t)\W_1(\mu,\bar{\mu}), \ \ \mu,\bar{\mu}\in\scr P_1(\R^d),t\geq 0.
%\end{align*}
This implies that
\begin{align*}
\W_1(\Gamma(\mu_1),\Gamma(\mu_2))
&\leq \inf_{t>\frac{\log c_0}{\lambda_0}}\frac{ G(t)}{1-c_0%\e^{-\lambda_0 t}
e^{-\lambda_0 t}
}\W_1(\mu_1,\mu_2), \quad \mu_1,\mu_2\in\scr P_1(\R^d).
\end{align*}
Since $\inf_{t>\frac{\log c_0}{\lambda_0}}\frac{ G(t)}{1-c_0 \e^{-\lambda_0 t}}\in[0,1)$, $\Gamma$ is a contraction mapping on $\P_1(\R^d)$.  So, it follows from the Banach fixed point theorem that, there exists a unique $\mu^\ast\in\scr P_1(\R^d)$ such that $\Gamma(\mu^\ast)=\mu^\ast$, which is the unique invariant probability measure to \eqref{GPS}.

(ii) Note that $\nu=(P_t^{\nu})^\ast\nu$,
$t\geq 0$. By the triangle inequality, we have
\begin{align}\label{abk}
\nonumber\W_1(P_t^\ast\eta,\nu)&\leq \W_1(P_t^\ast\eta,(P_t^{\nu})^\ast\eta)+\W_1((P_t^{\nu})^\ast\eta,\nu) \\
&\leq (H(t)+c_0 \e^{-\lambda_0 t})
\W_1(\eta,\nu), \quad \eta\in\scr P_1(\R^d),\,t\in[0,t_1].
\end{align}
Let $r_o=H(t_1)+c_0\e^{-\lambda_0 t_1}<1$. So, we arrive at
\begin{align*}
\W_1(P_{t_1}^\ast\eta,\nu)&\leq r_o\W_1(\eta,\nu),\ \ \eta\in\scr P_1(\R^d).
\end{align*}
For any $t\geq t_1,$ letting %$n_t=[\frac{t}{t_1}]$
$n_t$
be the integer part of $\frac{t}{t_1}$. By the semigroup property \eqref{semigroup property}, we conclude that for every $t\geq t_1$,
  \begin{equation*}
    \begin{split}
       \W_1(P_t^*\eta,\nu)&=\W_1(P_{t_1}^*P_{t-t_1}^*\eta,\nu)  \\
         & \leq r_o\W_1(P_{t-t_1}^*\eta,\nu)\\
         &\leq r_o^{n_t} \W_1(P_{t-t_1n_t }^*\eta, \nu)\\
         &\leq \e^{n_t\log r_o } \sup_{r\in[0,t_1]} \W_1(P_r^*\eta,\nu ),\ \ \eta\in\scr P_1(\R^d),
    \end{split}
  \end{equation*}
which together with \eqref{abk} completes the proof.
\end{proof}
\section{Applications on Brownian motion noise case}
\subsection{Non-degenerate and multiplicative cases}
Let
$\sigma\sigma^\ast\geq \alpha I_{d\times d}$ for some positive constant $\alpha$. %then
Then there exists a measurable function
$$\hat{\sigma}:\Rd \times \mathscr{P}(\Rd) \rightarrow\Rd \otimes \mathbb{R}^d$$ such that
$$(\sigma\sigma^*)(x,\mu)=\alpha I_{d\times d}+(\hat{\sigma}\hat{\sigma}^*)(x,\mu), \quad x\in \Rd, \mu \in \P(\Rd).$$
This means that the DDSDE
 \begin{equation}\label{NBM}
   \d X_t=b(X_t, \L_{X_t}) \d t +\sigma(X_t, \L_{X_t}) \d B_t
 \end{equation}
has the same distribution to
  \begin{equation}\label{NBM1}
   \d X_t=b(X_t, \L_{X_t}) \d t +\sqrt{\alpha}\d B_t^1+\hat{\sigma}(X_t, \L_{X_t}) \d B_t^2
 \end{equation}
 for two independent standard $d$-dimensional Brownian motions $B_t^i, i=1,2$ (cf. \cite{WFY 2023}). Hence, we will focus on investigating \eqref{NBM1} in %the subsequent analysis.
what follows.

 To establish the exponential ergodicity, we make the following assumptions.
 %\begin{enumerate}
 %\item[{\bf(A1)}] $b$ is continuous. Moreover, there exists a function $\phi \in C([0,\infty),\R)$ and some positive constants $C_1, C_2, K, K_0, K_1$ satisfying
%$$\phi(v)\leq Kv, \quad v\geq0,$$
%with \begin{equation}\label{Compared to r}
 % C_1 r \leq \psi(r):=\int_{0}^{r}\e^{-\int_{0}^{u}\frac{\phi(v)}{2\alpha} \, \d v}\int_{u}^{\infty}se^{\int_{0}^{s}\frac{\phi(v)}{2\alpha} \, \d v} \, \d s \, \d u \leq C_2 r, \quad r\geq0
%\end{equation}
%such that for any $x,y \in \Rd, \mu_1, \mu_2 \in \P_1(\Rd)$,
%\begin{align}\label{CTM}
%\|\hat{\sigma}(x, \mu_1)-\hat{\sigma}(y, \mu_2)\|\leq K_0(|x-y|+\W_1(\mu_1, \mu_2)), \ |b(0,\mu)|\leq K_0(1+\mu(|\cdot|)),
%\end{align}
%and
%\begin{equation}\label{monotonicity}
%\begin{split}
%&\<b(x,\mu_1)-b(y,\mu_2), x-y\>+\frac{1}{2}\|\hat{\sigma}(x, \mu_1)-\hat{\sigma}(y, \mu_2)\|_{\mathrm {HS}}^2\\
%&\leq \phi(|x-y|)|x-y|+K_1|x-y|\W_1(\mu_1, \mu_2)+K_1\W_1(\mu_1,\mu_2)^2,
%\end{split}
%\end{equation}
%where $\|\cdot\|_{HS}$ stands for the Hilbert--Schmidt norm and $\mu(|\cdot|):=\int_{\R^d}|x|\mu(\d x).$
% \end{enumerate}

\begin{enumerate}
\item[{\bf(A1)}] \begin{itemize}
\item [(i)] $b$ is continuous on $\Rd \times \P_1(\Rd),$ and there exists a constant $K_0>0$ such that for any $x,y \in \Rd$ and any $ \mu_1, \mu_2 \in \P_1(\Rd)$,
  $$ \|\hat{\sigma}(x, \mu_1)-\hat{\sigma}(y, \mu_2)\|_{\mathrm{HS}}\leq K_0(|x-y|+\W_1(\mu_1, \mu_2)).$$
  \item [(ii)] $b$ is locally bounded
  in $\Rd \times \P_1(\Rd)$, and there exists a constant $\widetilde{K}_0>0$ such that
   $$ \ |b(0,\mu)|\leq \widetilde{K}_0(1+\mu(|\cdot|)),\ \ \mu\in \P_1(\Rd).$$
   \item [(iii)] There exists a function $\phi \in C([0,\infty),\R)$ and some positive constants $C_1, C_2, K, \kappa$ %satisfying
   such that
$$\phi(v)\leq Kv, \quad v\geq0,$$
with \begin{equation}\label{Compared to r}
  C_1 r \leq \psi(r):=\int_{0}^{r}%\e
  \e^{-\int_{0}^{u}\frac{\phi(v)}{2\alpha} \, \d v}\int_{u}^{\infty}se^{\int_{0}^{s}\frac{\phi(v)}{2\alpha} \, \d v}\, \d s \d u \leq C_2 r, \quad r\geq0,
\end{equation}
and
\begin{equation}\label{monotonicity}
\begin{split}
&\<b(x,\mu_1)-b(y,\mu_2), x-y\>+\frac{1}{2}\|\hat{\sigma}(x, \mu_1)-\hat{\sigma}(y, \mu_2)\|_{\mathrm {HS}}^2\\
&\leq \phi(|x-y|)|x-y|+
%K_1  %%%%%%%%%% 为了前后符号统一 K_1都换成了\kappa
\kappa\W_1(\mu_1,\mu_2)^2,\ \ x,y \in \Rd, \mu_1, \mu_2 \in \P_1(\Rd).
\end{split}
\end{equation}
 \end{itemize}
\end{enumerate}

The main result of this section is contained in the following theorem.
\begin{thm}\label{main thm}
Assume {\bf(A1)} with $\psi^{\prime \prime}\leq 0$. Then there exists a %contant
constant
$\delta_0>0$ such that when $\kappa<\delta_0$, \eqref{NBM} has a unique
invariant probability measure $\mu^\ast \in \P_1(\Rd)$ and there exist constants $c,\lambda>0$ %such
satisfying that
\begin{equation}\label{speed}
  \W_1(P_t^\ast \eta, \mu^\ast)\leq ce^{-\lambda t} \W_1(\eta,\mu^\ast),\ \ t\geq 0, \,\eta \in \P_1(\R^d).
\end{equation}
\end{thm}
\begin{rem}
(1) It is known that \eqref{NBM} is well-posed for distributions in $\P_1(\Rd)$ under assumption \textbf{(A1)}; see \cite[Theorem 3.3.1]{WFY 2024} or \cite{Huang 2019} for details.

(2) The assumption that $\kappa$ is %less
smaller
than a certain constant is reasonable since there will be phase transition when $\kappa$ is big enough. In \cite{Dawson 1983}, phase transition was first established for an equation with a specific double-well confinement on the real line. This example illustrates that the appearance of an interaction potential may destroy the ergodicity of the original system and multiple stationary distributions may %produce
appear. We provide %more examples
a different model, whose diffusion coefficient depends on the measure variable,
to illustrate this point.
\end{rem}

\begin{exa}
Let $\epsilon>0$. We consider the following SDE on $\R$:
$$\d X_t=(-X_t+\epsilon )\d t +\epsilon\d B_t.$$
%For the classical SDE mentioned above, given that the drift of the equation is uniformly dissipative and the noise is non-degenerate, then the associated Markov process possesses a unique invariant probability measure, as supported by Lemma \ref{lemma of sde}.
Since the drift of the equation is uniformly dissipative and the noise is non-degenerate, the associated Markov process possesses a unique invariant probability measure, as supported by Lemma \ref{lemma of sde} below.

It turns out that the appearance of a distribution-dependent term can significantly influence the ergodic properties. To illustrate this, we consider
 the following McKean--Vlasov SDE
  \begin{equation}\label{distribution dependent}
\d X_t=(-X_t+\epsilon \E f(X_t))\d t +\epsilon \E f(X_t)\d B_t,
\end{equation}
where $f(x)=|x|+1$ for all $x\in \R.$

The stationary distribution of %the DDSDE
\eqref{distribution dependent}
can be characterized by its corresponding equilibrium Fokker--Planck equation, which is given by
\begin{equation*}
\begin{cases}
 \frac{\partial}{\partial x}\left[(-x+\epsilon a)p(a,x)\right]=\frac{\epsilon^2 a^2}{2}\frac{\partial^2}{\partial x^2} p(a,x),\\
a=m(a):=\int_{\mathbb{R}} f(x)p(a,x) \, \d x ,
\end{cases}
\end{equation*}
where $p(a,x)$ %denote
denotes
the probability density function of the equilibrium probability distribution corresponding to equation \eqref{distribution dependent}.

 Note that if there exists a constant $a\geq 1$ such that $m(a)=a$, then
$$ p(a,x)= \frac{1}{\e a \epsilon\sqrt{\pi}}\e^{-\frac{x^2}{a^2\epsilon^2}+\frac{2x}{a\epsilon}},\quad x\in\R.$$
Consequently, there is a one-to-one correspondence between equilibrium probability distributions and solutions of the equation
\begin{equation*}
  m(a)=a.
\end{equation*}
Straightforward calculations yield
\begin{equation*}
  \begin{split}
  m(a) &=\int_{\R} f(x)p(a,x) \, \d x \\
  &=\frac{a\epsilon}{\e\sqrt{\pi}}\int_{\R} |x|\e^{-x^2+2x} \, \d x +1 \\
  &=\frac{a\epsilon}{\sqrt{\pi}}\int_{\R} |x+1|\e^{-x^2} \d x+1\\
  &=\frac{a\epsilon}{\sqrt{\pi}}\mathfrak{c}+1,
  \end{split}
\end{equation*}
where $\mathfrak{c}:=\int_{\R} |x+1|\e^{-x^2} \d x>0$. %which implies that
As a consequence,
when $\epsilon<\frac{\sqrt{\pi}}{\mathfrak{c}},$ there exists a unique invariant probability measure for equation  \eqref{distribution dependent}. Conversely, when $\epsilon\geq \frac{\sqrt{\pi}}{\mathfrak{c}},$ the equation \eqref{distribution dependent} has no invariant probability measure. %Here $c:=\int_{\R} |x+1|e^{-x^2} \d x$ is a positive constant.
\end{exa}

Now we present a concrete example of the function $\phi$ appearing in {\bf(A1)}(iii),  which is typically associated with dissipative conditions at long distances.
\begin{cor} Assume that (i), (ii) and \eqref{monotonicity} in {\bf(A1)} hold for $$\phi(v)=\left\{
  \begin{array}{ll}
l_1v, & \hbox{$v\leq r_0$,}\\
    \left\{-\frac{l_1+l_2}{r_0}(v-r_0)+l_1\right\}v, &\hbox{$r_0< v\leq 2r_0$,}\\
    -l_2v, & \hbox{$r>2r_0$}
  \end{array}
\right.
$$
with some positive constants $l_1, l_2, r_0.$ Then (iii) in {\bf(A1)} holds and consequently, the assertions in Theorem \ref{main thm} hold.
\end{cor}
 \begin{proof} According to  \cite[(3.23)]{HYY}, {\bf(A1)} holds for $\phi$ with
$$C_1=\frac{2\alpha}{l_2},\quad C_2=\int_0^\infty s%\e
e^{\frac{1}{2\alpha}\int_0^s\phi(v)\d v}\d s,\quad K=l_1.$$
Moreover, by \cite[Page 20]{HYY}, it holds that $\psi''\leq 0$. The proof is completed by Theorem \ref{main thm}.
\end{proof}

In what follows, we aim to prove Theorem \ref{main thm}. To this end, we first consider the decoupled SDE with a fixed distribution $\mu \in \P_1(\Rd)$:
\begin{equation}\label{sde}
  \d X_t^{\mu}=b(X_t^{\mu}, \mu) \d t+\sqrt{\alpha}\d B_t^1+\hat{\sigma}(X_t^{\mu}, \mu) \d B_t^2.
\end{equation}
Recall that $(P_t^\mu)^\ast\eta$ stands for the distribution to \eqref{sde} with initial distribution $\eta\in\scr P_1(\R^d)$. The following lemma is crucial.
\begin{lem}\label{lemma of sde}
%Let
Suppose that {\bf (A1)} holds with $\psi^{\prime \prime}\leq0$. Then \eqref{sde}
   has a unique invariant probability measure $\Gamma(\mu)\in \P_1(\Rd)$. Moreover, there exist constants $c_0>1, \lambda_0>0$ independent of $\mu$ such that
  \begin{equation}\label{contractivity}
    \W_1((P_t^\mu)^ \ast \eta_1,(P_t^\mu)^\ast \eta_2 )\leq c_0 \e^{-\lambda_0 t}\W_1(\eta_1,\eta_2),\quad t\geq 0,\, \eta_1,\eta_2 \in \P_1(\Rd).
  \end{equation}
\end{lem}

To prove Lemma \ref{lemma of sde}, we introduce the Wasserstein distance induced by $\psi$. Define
$$\W_\psi(\mu,\nu):=\inf_{\pi \in \mathscr{C}(\mu, \nu)}\int_{\Rd\times \Rd} \psi(|x-y|) \pi(\d x, \d y), \quad \mu,\nu\in \P_1(\Rd).$$
Since $\psi'\geq 0$, $\psi^{\prime \prime}\leq 0$ and $C_1r\leq\psi(r)\leq C_2 r$, it is easy to see that $(\P_1(\Rd),\W_\psi)$ is a Polish space.

For any $\eta_1, \eta_2 \in \P_1(\Rd)$, let $X_0^\mu, Y_0^\mu$ be $\F_0$-measurable $\R^d$-valued random variables such that
$$\L_{X_0^\mu}=\eta_1, \quad \L_{Y_0^\mu}=\eta_2,$$
and
$$\E\psi(|X_0^\mu-Y_0^\mu|)=\W_\psi(\eta_1, \eta_2).$$

To prove Lemma \ref{lemma of sde}, we use the method of coupling by reflection and we introduce
$$u(x,y):=\frac{x-y}{|x-y|},\quad x\neq y \in \Rd.$$

Consider
\begin{equation}\label{reflet sde2}
  \d Y_t^\mu=b(Y_t^\mu, \mu) \d t+\sqrt{\alpha}\{I_{d\times d}-2u(X_t^\mu,Y_t^\mu)\otimes u(X_t^\mu,Y_t^\mu)\mathds{1}_{\{t<\tau\}}\} \d B_t^1+\hat{\sigma}(Y_t^\mu, \mu) \d B_t^2,
\end{equation}
where
$$\tau:=\inf\{t\geq0;\ X_t^\mu=Y_t^\mu\}$$
is the so-called coupling time.
\begin{proof}[Proof of Lemma \ref{lemma of sde}]
 (1) \underline{Proof of \eqref{contractivity}.} For any $r\geq0$, it holds that
\begin{equation*}\begin{split}
&\psi^\prime(r)=\e^{-\int_{0}^{r}\frac{\phi(v)}{2\alpha} \, \d v}\int_{r}^{\infty} s\e^{\int_{0}^{s}\frac{\phi(v)}{2\alpha} \, \d v} \, \d s \geq 0,\\
 &\psi^{\prime \prime}(r)=-\frac{1}{2\alpha}\phi(r)\psi^\prime(r)-r.
\end{split}\end{equation*}
These
%This
together with \eqref{Compared to r} %implies
imply
\begin{equation}\label{ode}
  2\alpha \psi^{\prime \prime}(r)+\phi(r)\psi^\prime(r)\leq -\frac{2\alpha}{C_2}\psi(r).
\end{equation}
Applying the It\^{o}-Tanaka formula for \eqref{sde} and \eqref{reflet sde2}, we arrive at
\begin{equation*}
\begin{split}
    \d |X_t^\mu-Y_t^\mu| &\leq \phi(|X_t^\mu-Y_t^\mu|) \d t +2\sqrt{\alpha}\<u(X_t^\mu, Y_t^\mu), \d B_t^1\>\\
     &+\<u(X_t^\mu, Y_t^\mu),(\hat{\sigma}(X_t^\mu, \mu)-\hat{\sigma}(Y_t^\mu, \mu))\d B_t^2\>, \quad t<\tau.
\end{split}
\end{equation*}
By It\^{o}'s formula and the assumption $\psi''\leq 0$, we obtain
\begin{equation*}
  \begin{split}
     \d \psi(|X_t^\mu-Y_t^\mu|) & \leq\{\psi^{\prime}(|X_t^\mu-Y_t^\mu|)\phi(|X_t^\mu-Y_t^\mu|)+2\alpha \psi^{\prime \prime}(|X_t^\mu-Y_t^\mu|)\} \d t\\
     &+\psi^{\prime}(|X_t^\mu-Y_t^\mu|)\{2\sqrt{\alpha}\<u(X_t^\mu, Y_t^\mu), \d B_t^1\>\}\\
     &+\psi^{\prime}(|X_t^\mu-Y_t^\mu|)\{\<u(X_t^\mu, Y_t^\mu),(\hat{\sigma}(X_t^\mu, \mu)-\hat{\sigma}(Y_t^\mu, \mu))\d B_t^2\>\}\\
       & \leq -\frac{2\alpha}{C_2}\psi(|X_t^\mu-Y_t^\mu|) \d t
       +2\sqrt{\alpha}\psi^{\prime}(|X_t^\mu-Y_t^\mu|)\<u(X_t^\mu, Y_t^\mu), \d B_t^1\>\\
       &+\psi^{\prime}(|X_t^\mu-Y_t^\mu|)\<u(X_t^\mu, Y_t^\mu),(\hat{\sigma}(X_t^\mu, \mu)-\hat{\sigma}(Y_t^\mu, \mu))\d B_t^2\>,\quad t<\tau.
  \end{split}
\end{equation*}
Applying Gr\"{o}nwall's inequality, we have
 $$\E%\e
 \e^{\frac{2\alpha}{C_2}t\wedge \tau}\psi(|X_{t\wedge \tau}^\mu-Y_{t\wedge \tau}^\mu|)\leq \E\psi(|X_0^\mu-Y_0^\mu|), \quad t\geq 0.$$
This combined with the fact that $X_t^\mu=Y_t^\mu$, $t\geq\tau$, gives
 $$\E \e^{\frac{2\alpha}{C_2}t}\psi(|X_{t}^\mu-Y_{t}^\mu|)\leq \E%\e
 \e^{\frac{2\alpha}{C_2}t\wedge \tau}\psi(|X_{t\wedge \tau}^\mu-Y_{t\wedge \tau}^\mu|)\leq \E\psi(|X_0^\mu-Y_0^\mu|), \quad t\geq 0,$$
which implies
 \begin{equation*}
  \begin{split}
     &\W_\psi((P_t^\mu)^*\eta_1,(P_t^\mu)^*\eta_2) \leq \E\psi(|X_t^\mu-Y_t^\mu|) \\
     &\leq \e^{-\frac{2\alpha t}{C_2}}\E\psi(|X_0^\mu-Y_0^\mu|)=%\e
     \e^{-\frac{2\alpha t}{C_2}} \W_\psi(\eta_1,\eta_2),\quad t\geq0.
  \end{split}
 \end{equation*}
Combining this with \eqref{Compared to r}, we deduce that there exist constants $c_0>1$ and $\lambda_0>0$ such that for any $ t\geq 0,\, \eta_1,\eta_2 \in \P_1(\Rd)$,
 $$ \W_1((P_t^\mu)^ \ast \eta_1,(P_t^\mu)^\ast \eta_2 )\leq c_0 e^{-\lambda_0 t}\W_1(\eta_1,\eta_2).$$

 (2) \underline{Existence and uniqueness of $\Gamma(\mu)$.}  The uniqueness of the invariant probability for $((P_t^{\mu})^\ast)_{t\geq 0}$ follows immediately  from \eqref{contractivity}. In fact, for two invariant probability measures $\xi_1,\xi_2\in\scr P_1(\R^d)$ of $((P_t^{\mu})^\ast)_{t\geq 0}$, it follows form \eqref{contractivity} that
 $$\W_1(\xi_1,\xi_2)\leq \inf_{t\geq 0}\W_1((P_t^{\mu})^\ast\xi_1,(P_t^{\mu})^\ast\xi_2)=0. $$
 It suffices to show that $((P_t^{\mu})^\ast)_{t\geq 0}$ has a unique invariant probability measure $\Gamma(\mu)\in \P_1(\Rd)$ from \eqref{contractivity}. To this end, we intend to apply the idea in \cite{WFY 2018}.

 Let $\delta_0$ be the Dirac measure at 0.
  By the semigroup property $(P_{t+s}^{\mu})^*=(P_{t}^{\mu})^*(P_{s}^{\mu})^*$ for $s,t \geq 0$ and \eqref{contractivity}, we have
 $$\W_1((P_{t}^{\mu})^*\delta_0,( P_{t+s}^{\mu})^*\delta_0)\leq c_0%\e
 \e^{-\lambda_0 t} \W_1(\delta_0, (P_{s}^{\mu})^*\delta_0), \quad s,t \geq 0.$$
 For any $s\geq 0$, let $n_s$ be the integer part of $s$. %, according
  According
  to the  triangle inequality, we have
 \begin{equation*}
   \begin{split}
      \sup_{s\geq 0}\W_1(\delta_0, (P_{s}^{\mu})^*\delta_0) &\leq \sup_{s\geq 0}\sum_{m=0}^{n_s-1}\W_1((P_m^{\mu})^*\delta_0,%P_{m+1}^{\gamma})^*\delta_0
      (P_{m+1}^{\mu})^*\delta_0)+\sup_{s\geq 0}\W_1((P_{n_s}^{\mu})^*\delta_0, (P_s^{\mu})^*\delta_0) \\
        & \leq \sup_{s\geq 0}\sum_{m=0}^{n_s-1} c_0%\e
        \e^{-\lambda_0 m} \W_1(\delta_0, (P_1^{\mu})^*\delta_0)+\sup_{s\geq 0}c_0\e^{-\lambda_0 n_s}\sup_{r\in [0,1]} [(P_r^{\mu})^*\delta_0](|\cdot|)\\
        &< \infty.
   \end{split}
 \end{equation*}
Therefore, as $t$ tends to infinity, $ (P_t^{\mu})^*\delta_0 $ converges to a probability measure $\Gamma(\mu) \in \P_1(\Rd)$ in $\W_1$, which is the invariant probability measure of $((P_t^{\mu})^{\ast})_{t\geq 0}.$ Indeed, for any $s\geq 0,$ it follows from \eqref{contractivity} that
\begin{equation*}
   \begin{split}\W_1((P_s^{\mu})^*\Gamma(\mu), \Gamma(\mu))&=\lim_{t\to \infty}\W_1((P_s^{\mu})^*\Gamma(\mu),(P_s^{\mu})^*(P_t^{\mu})^*\delta_0)\\
   &\leq \lim_{t\to \infty}\W_1(\Gamma(\mu),(P_t^{\mu})^*\delta_0)=0.
      \end{split}
 \end{equation*}

  The proof is completed.
\end{proof}
 Now, we are in a position to prove Theorem \ref{main thm}.
 \begin{proof}[Proof of Theorem \ref{main thm}] (1) \underline{Existence and uniqueness of $\mu^\ast$}.

By combining \eqref{contractivity} with the conclusion of Theorem \ref{ABC}(i), it is sufficient to estimate $\W_1((P_t^{\mu_1})^*\Gamma(\mu_2),\Gamma(\mu_2)) $ for all $\mu_1,\mu_2 \in \P_1(\Rd),\, t\geq 0.$ For this purpose, we construct the following synchronous coupling.

 Let $X^1, X^2$ be the solutions to the following SDEs:
 \begin{equation*}
   \begin{split}
    & \d X_t^1=b(X_t^1, \mu_1)\d t+\sqrt{\alpha}\d B_t^1+\hat{\sigma}(X_t^1, \mu_1) \d B_t^2,\\
        & \d X_t^2=b(X_t^2, \mu_2)\d t+\sqrt{\alpha}\d B_t^1+\hat{\sigma}(X_t^2, \mu_2) \d B_t^2
   \end{split}
 \end{equation*}
with
 $X_0^1=X_0^2$ having distribution $\Gamma(\mu_2)\in\scr P_1(\R^d)$. Note that
 $$\L_{X_t^1}=(P_t^{\mu_1})^\ast \Gamma(\mu_2), \quad  \L_{X_t^2}=\Gamma(\mu_2).$$
 By It\^o's formula and {\bf (A1)}, we have
 \begin{equation*}
   \begin{split}
      &\d |X_t^1-X_t^2|^2 \\
      &\leq \{2\phi(|X_t^1-X_t^2|)|X_t^1-X_t^2|+2\kappa\W_1(\mu_1,\mu_2)^2\} \d t\\
      &+ 2 \<X_t^1-X_t^2, \ \ \big(\hat{\sigma}(X_t^1,\mu_1)-\hat{\sigma}(X_t^2,\mu_2)\big)\d B_t^2\>\\
      &\leq \{2K|X_t^1-X_t^2|^2+2\kappa\W_1(\mu_1,\mu_2)^2\} \d t\\
      &+ 2 \<X_t^1-X_t^2, \ \ \big(\hat{\sigma}(X_t^1,\mu_1)-\hat{\sigma}(X_t^2,\mu_2)\big)\d B_t^2\>, \quad t\geq0.
       \end{split}
 \end{equation*}
By a standard stopping time technique, it follows from Gr\"{o}nwall's %lemma
inequality that
 \begin{equation*}
   \begin{split}
      \E|X_t^1-X_t^2|^2 \leq \frac{\kappa}{K}(\e^{2Kt}-1)\W_1(\mu_1,\mu_2)^2, \quad t\geq 0,
        \end{split}
 \end{equation*}
which implies
\begin{equation*}
   \begin{split}
   \W_1((P_t^{\mu_1})^*\Gamma(\mu_2),\Gamma(\mu_2)) &\leq \left(\E|X_t^1-X_t^2|^2\right)^{\frac{1}{2}} \\
        & \leq \sqrt{\frac{\kappa}{K}(\e^{2Kt}-1)}\, \W_1(\mu_1,\mu_2)=:G(t)\W_1(\mu_1,\mu_2),\quad t\geq0.
\end{split}
 \end{equation*}
Set
$$\delta_1:=\left(\inf_{t>\frac{\log c_0}{\lambda_0}}\frac{\sqrt{\frac{1}{K}(\e^{2Kt}-1)}}{1-c_0\e^{-\lambda_0 t}
}\right)^{-2},$$
where $c_0, \lambda_0$ are the constants from \eqref{contractivity}.
Then using Theorem \ref{ABC}(i), we conclude that when $\kappa<\delta_1,$
 $(P_t^\ast)_{t\geq 0}$ has a unique invariant probability measure $\mu^\ast.$

(2) \underline{Proof of \eqref{speed}}. We only need to estimate  $\W_1(P_t^{*}\eta,(P_t^{\mu^\ast})^*\eta)$ for $\eta\in\scr P_1(\R^d)$. Let $Y^1, Y^2$ solve the following SDEs
 \begin{equation*}
   \begin{split}
     &\d Y_t^1=b(Y_t^1, \mu^\ast)\d t+\sqrt{\alpha}\d B_t^1+\hat{\sigma}(Y_t^1, \mu^\ast) \d B_t^2,\\
        & \d Y_t^2=b(Y_t^2, P_t^{*}\eta)\d t+\sqrt{\alpha}\d B_t^1+\hat{\sigma}(Y_t^2, P_t^{*}\eta) \d B_t^2
   \end{split}
 \end{equation*}
with $Y_0^1=Y_0^2$ satisfying $\L_{Y_0^1}=\eta$. Note that
$$\L_{Y_t^1}= (P_t^{\mu^\ast})^*\eta, \quad \L_{Y_t^2}=P_t^{*}\eta.$$
 By It\^{o}'s formula and {\bf(A1)}, we have
  \begin{equation*}
    \begin{split}
       &\d |Y_t^1-Y_t^2|^2 \\
       &\leq\{2\phi(|Y_t^1-Y_t^2|)|Y_t^1-Y_t^2| +2\kappa\W_1(\mu^\ast,P_t^*\eta)^2\} \d t +\d M_t\\
         & \leq 2K|Y_t^1-Y_t^2|^2\d t+ 4\kappa\left[c_0^2\e^{-2\lambda_0 t}\W_1(\mu^\ast, \eta)^2+\W_1(P_t^*\mu,(P_t^{\mu^\ast})^*\eta)^2\right]\d t +\d M_t\\
    \end{split}
  \end{equation*}
  for some continuous local martingale $(M_t)_{t\geq0}$ with $M_0=0$, where in the last step, we use the following inequality
  \begin{equation*}
    \begin{split}
       \W_1(P_t^\ast\eta,\mu^\ast) &\leq \W_1(P_t^*\eta, (P_t^{\mu^\ast})^\ast \eta)+\W_1((P_t^{\mu^\ast})^\ast \eta, \mu^\ast) \\
         & \leq \W_1(P_t^*\eta, (P_t^{\mu^\ast})^\ast \eta)+c_0\e^{-\lambda_0 t}\W_1(\eta,\mu^\ast),\quad  t\geq0.
    \end{split}
  \end{equation*}
 By Gr\"{o}nwall's inequality, we obtain
  $$\E |Y_t^1-Y_t^2|^2 \leq \W_1(\mu^\ast, \eta)^2 4\kappa c_0^2\int_0^t\e^{(4\kappa+2K)(t-s)}\e^{-2\lambda_0 s}\d s, \quad t\geq 0,$$
  which implies
\begin{align}\label{w1erg}\W_1((P_t^{\mu^\ast})^*\eta,P_t^*\eta) \leq c_0\sqrt{4\kappa\int_0^t\e^{(4\kappa+2K)(t-s)}\e^{-2\lambda_0 s}\d s}\, \W_1(\mu^\ast, \eta), \quad t\geq0.
\end{align}
  Let
  $$H(t)=c_0\sqrt{4\kappa \int_0^t\e^{(4\kappa+2K)(t-s)}\e^{-2\lambda_0 s}\d s},\quad t\geq0,$$
  and
  $$\delta_2=\sup\left\{\kappa>0: \inf_{t>\frac{\log c_0}{\lambda_0}} (H(t)+c_0\e^{-\lambda_0 t})<1\right\},$$
  where $c_0, \lambda_0$ are the constants from \eqref{contractivity}. Set $\delta_0:=\min\{\delta_1, \delta_2\}$.
  Then, by Theorem \ref{ABC}(ii), when $\kappa<\delta_0$, there exist constants $c>1$ and $\lambda>0$ such that \eqref{speed} holds.

 \end{proof}
 \subsection{%Applications %in Second order system on
Second order system }

 In this part, we consider the following second order system
 \begin{equation}\label{Langevin}
\begin{cases}
\d X_t=Y_t\d t, \\
\d Y_t=-\gamma Y_t\d t+ b(X_t, \L_{(X_t, Y_t)}) \d t +\sigma (\L_{(X_t, Y_t)}) \d B_t,
\end{cases}
\end{equation}
where the constant $\gamma>0$ represents the friction coefficient, $$b:\Rd \times \mathscr{P}(\mathbb{R}^{2d})
 \rightarrow
   \Rd, \quad \sigma:\mathscr{P}(\mathbb{R}^{2d})
    \rightarrow
    \mathbb{R}$$
    are measurable,
%and the constant $\gamma>0$ represents the friction coefficient,
and $(B_t)_{t\geq0}$ is a  $d$-dimensional standard  Brownian motion.

For $\mu \in \scr P_1(\R^{2d})$, the time-homogeneous decoupled SDEs are given by
 \begin{equation*}
\begin{cases}
\d X_t^\mu=Y_t^\mu\d t, \\
\d Y_t^\mu=-\gamma Y_t^\mu\d t+ b(X_t^\mu, \mu) \d t +\sigma (\mu) \d B_t.
\end{cases}
\end{equation*}

To derive the exponential ergodicity in $\W_1$ for \eqref{Langevin}, we make the following assumptions.
\begin{enumerate}
\item[{\bf (A2)}]
(i) {(\emph{Continuity)}} There exist constants $%L_b^1
L_b,%%%%下同,
\kappa>0$ such that for any $x,
%\bar{x} 为了前后符号一致,\bar{x}都换成了y,后面也是
y \in \R^d,\, \mu_1, \mu_2 \in \P_1(\R^{2d})$,
 \begin{equation*}\begin{split}
 &|b(x,\mu_1)-b(y,\mu_2)|\leq L_b|x-y|+ \sqrt{\kappa} \W_1(\mu_1, \mu_2),\\
 & d|\sigma(\mu_1)-\sigma(\mu_2)|^2 \leq \kappa\W_1(\mu_1,\mu_2)^2.
 \end{split}\end{equation*}

(ii) {(\emph{Partially dissipative condition})} There exist constants $K_1, R>0$ such that for any $x,y\in \R^d, \,\mu\in \P_1(\R^{2d})$,
\begin{align*}
&\<b(x,\mu)-b(y,\mu), x-y\>\leq -K_1|x-y|^2 , \quad |x-y|\geq R.
\end{align*}
(iii) {(\emph{Non-degeneracy})} There exists a constant $ \delta\geq 1$ such that
$$ \delta^{-1}\leq \sigma^2\leq \delta.$$
\end{enumerate}

According
to \cite[Theorem 3.3.1]{WFY 2024},
under %Assumption
assumption
\textbf{(A2)}, the DDSDE \eqref{Langevin} is well-posed in $\P_1(\R^{2d})$.

\begin{thm}\label{dBM} Suppose that {\bf(A2)} holds with
\begin{equation}\label{dBM-1}
(K_1+L_b)\gamma^{-2}\leq \frac{K_1}{2(K_1+L_b)}.
 %2(K_1+L_b)^2\leq  K_1\gamma^{2}
\end{equation}
Then there exists a constant $\delta_0>0$ such that when $\kappa<\delta_0$, \eqref{Langevin} has a unique invariant probability measure $\mu^\ast\in\P_1(\R^{2d})$ and there exist constants $c,\lambda>0$ such that
\begin{equation*}
  \W_1(P_t^*\eta, \mu^\ast)\leq c%\e^{-\lambda t}
   \e^{-\lambda t}
  \W_1(\eta, \mu^\ast),\quad \eta\in\scr P_1(\R^{2d}),\, t\geq 0.
\end{equation*}
\end{thm}
\begin{rem} Compared with the recent result on ergodicity in \cite[Theorem 12]{Schuh 2024}, in Theorem \ref{dBM}, the coefficient $b$ is allowed to be dependent on the distribution of $Y_t$.  Additionally, in contrast to their assumption of a constant diffusion coefficient, our framework allows for a distribution dependent diffusion coefficient.
\end{rem}
\begin{proof}[Proof of Theorem \ref{dBM}]
Let $\mu \in \scr P_1(\R^{2d})$ and $(\bar{X}_t^\mu,\bar{Y}_t^\mu)=\sigma(\mu)^{-1}(X_t^\mu,Y_t^\mu)$, $t\geq0$. Then it holds
 \begin{equation}\label{barXY}
\begin{cases}
\d \bar{X}_t^\mu=\bar{Y}_t^\mu\d t, \\
\d \bar{Y}_t^\mu=-\gamma \bar{Y}_t^\mu\d t+ \bar{b}(\bar{X}_t^\mu, \mu)\d t+ \d B_t,
\end{cases}
\end{equation}
where
$$\bar{b}(x,\mu):=\sigma(\mu)^{-1}b(\sigma(\mu)x,\mu), \quad x\in \Rd,\, \mu \in \scr P_1(\R^{2d}).$$
According to \textbf{(A2)}, for any $x,y \in \Rd, \mu \in \scr P_1(\R^{2d}),$ it holds
\begin{align}\label{bar-b}|\bar{b}(x,\mu)-\bar{b}(y,\mu)|\leq L_b |x-y|,
\end{align}
and
\begin{align}\label{pdi}
\nonumber\<\bar b(x,\mu)-\bar b(y,\mu),x-y\>&=\sigma(\mu)^{-1}\<b(\sigma(\mu)x,\mu)-b(\sigma(\mu)y,\mu),x-y\>\\
&\leq -K_1|x-y|^2,\quad |x-y|\geq \sqrt{\delta}R.
\end{align}
Letting $\bar{g}(x,\mu)=K_1x+\bar{b}(x,\mu)$, we have
\begin{align}\label{bde}\bar{b}(x,\mu)=-K_1x+\bar{g}(x,\mu).
\end{align} Then \eqref{bar-b} implies that
and for any $x,y \in \Rd,\,\mu\in \P_1(\R^{2d})$,
\begin{align}\label{gli}|\bar{g}(x,\mu)-\bar{g}(y,\mu)|\leq (K_1+L_b)|x-y| .
\end{align}
Moreover, \eqref{pdi} gives
\begin{align}\label{gf0}
&\<\bar{g}(x,\mu)-\bar{g}(y,\mu), x-y\>\leq 0,\quad |x-y|\geq \sqrt{\delta}R,\,\mu \in \P_1(\R^{2d}) .
\end{align}

%Let $u=\frac{1}{2\gamma}$, %$b=2\gamma\bar{b}(\cdot,\mu)$,
%%%b和前面的b混淆了，需要引入新的记号
%{\color{red}$\hat{b}(\cdot, \mu)=2\gamma\bar{b}(\cdot,\mu), \mu\in \P_1(\Rd)$,} $K=2\gamma K_1 I_{d\times d}$, $g=2\gamma\bar{g}(\cdot,\mu)$, $L_K=2\gamma K_1$, $L_g=2\gamma(K_1+L_b)$.
For $\mu\in \P_1(\Rd)$, let $\hat{b}(\cdot, \mu)=2\gamma\bar{b}(\cdot,\mu)$ and  $g(\cdot,\mu)=2\gamma\bar{g}(\cdot,\mu)$.  Let $u=\frac{1}{2\gamma}$, $K=2\gamma K_1 I_{d\times d}$ with eigenvalue $L_K=2\gamma K_1$, and $L_g=2\gamma(K_1+L_b)$ is the Lipschitz constant of $g$. Then, \eqref{barXY} is equivalent to
\begin{equation*}
\begin{cases}
\d\bar X_t^\mu=\bar Y_t^\mu\d t, \\
\d \bar Y_t^\mu=-\gamma \bar Y_t^\mu \d t+u\hat{b}(\bar X_t^\mu, \mu)\d t+ \sqrt{2\gamma u} \d B_t.
\end{cases}
\end{equation*}
From assumption \eqref{dBM-1}, %It is not difficult to see that
% $$(K_1+L_b) \gamma^{-2}< \frac{K_1}{2(K_1+L_b)},$$
%which implies
it is easy to see that
$$L_g u\gamma^{-2}<\frac{L_K}{2L_g}.$$
Then, by \eqref{bde}-\eqref{gf0} and applying \cite[Theorem 5]{Schuh 2024}, there exist constants $c_0, \lambda_0>0$ independent of $\mu$ such that
\begin{equation}\label{w-contr}
\W_1((\bar{P}_t^\mu)^\ast \eta_1,(\bar{P}_t^\mu)^\ast\eta_2)\leq c_0 \e^{-\lambda_0 t}
\W_1(\eta_1,\eta_2),\quad \eta_1,\eta_2 \in \P_1(\R^{2d}),
\end{equation}
where $(\bar{P}_t^\mu)^\ast \eta:=%\L_{(\bar{X}_t^\mu,(\bar{Y}_t^\mu)}
\L_{(\bar{X}_t^\mu,\bar{Y}_t^\mu)}$
for $\L_{(\bar{X}_0^\mu,\bar{Y}_0^\mu)}=\eta\in\P_1(\R^{2d}),\, t\geq 0.$ Consequently,
for any $\mu\in\scr P_1(\R^{2d})$, there exists a unique invariant probability measure $\bar{\Gamma}(\mu)\in\scr P_1(\R^{2d})$ to $(\bar{X}_t^\mu, \bar{Y}_t^\mu)$.

For any Borel measurable set $A\subseteq\R^{2d}$ and any $\eta\in \P_1(\R^{2d})$, define
$$\Gamma(\mu)(A):=\bar{\Gamma}(\mu)(\sigma(\mu)A),$$
and
$$\eta^\mu(A):=\eta(\sigma(\mu)A).$$
Then for any $\mu\in\scr P_1(\R^{2d})$, and $\bar{\Gamma}(\mu)\in\P_1(\R^{2d})$,  %%%删除\{...\}
  $\Gamma(\mu)$ is the unique invariant probability measure of \eqref{Langevin}.
Recall $(P_t^\mu)^*\eta=\L_{(X_t^\mu, Y_t^\mu)}$ with $\L_{(X_0^\mu, Y_0^\mu)}=\eta\in\scr P_1(\R^{2d})$.
For fixed $t\geq0$ and $\eta\in\scr P_1(\R^{2d})$, let $(U,V)$ be an optimal coupling of $((P_t^\mu)^*\eta,\Gamma(\mu))$. It is easy to see that $\sigma(\mu)^{-1}(U,V)$ turns out to be an optimal coupling of $((\bar{P}_t^\mu)^\ast\eta^\mu,\bar{\Gamma}(\mu))$, and hence,
$$\W_1((P_t^\mu)^\ast \eta, \Gamma(\mu)) =|\sigma(\mu)|\W_1((\bar{P}_t^\mu)^\ast\eta^\mu,\bar{\Gamma}(\mu)).$$
Thus, combining this with \eqref{w-contr}, we have
\begin{equation}\label{gap}
  \begin{split}
     \W_1((P_t^\mu)^\ast \eta, \Gamma(\mu)) %&=|\sigma(\mu)|\W_1((\bar{P}_t^\mu)^\ast\eta^\mu,\bar{\Gamma}(\mu)) \\
       & \leq |\sigma(\mu)|c_0\e^{-\lambda_0 t}\W_1(\eta^\mu,\bar{\Gamma}(\mu))\\
      &= |\sigma(\mu)||\sigma(\mu)^{-1}|c_0\e^{-\lambda_0 t}\W_1(\eta,\Gamma(\mu))\\
       &\leq c_0\e^{-\lambda_0 t}\W_1(\eta,\Gamma(\mu)),\quad \eta\in\scr P_1(\R^{2d}),\, t\geq0.
  \end{split}
\end{equation}
Next, for any $\mu_i\in\scr P_1(\R^{2d}),i=1,2$, let $(X_t^i,Y_t^i)$ be the solution to the following SDEs:
  \begin{equation*}
\begin{cases}
\d X_t^i=Y_t^i\d t, \\
\d Y_t^i=-\gamma Y_t^i\d t+b(X_t^i,\mu_i)\d t+\sigma(\mu_i) \d B_t %\quad i=1,2,
\end{cases}
\end{equation*}
for $(X_0^1,Y_0^1)=(X_0^2,Y_0^2)$ having distribution $\Gamma(\mu_2).$

Then it follows from  It\^{o}'s formula and \textbf{(A2)} that
\begin{equation*}
\begin{split}
   &\d[|X_t^1-X_t^2|^2+|Y_t^1-Y_t^2|^2] \\
   &=2\<X_t^1-X_t^2, Y_t^1-Y_t^2\>\d t +2\<Y_t^1-Y_t^2, \d (Y_t^1-Y_t^2)\>+d|\sigma(\mu_1)-\sigma(\mu_2)|^2 \d t \\
     & \leq (L_b+2)[|X_t^1-X_t^2|^2+|Y_t^1-Y_t^2|^2] \d t+2\kappa\W_1(\mu_1,\mu_2)^2 \d t +\d M_t
\end{split}
\end{equation*}
for some continuous local martingale $(M_t)_{t\geq0}$ with $M_0=0$. According to Gr\"{o}nwall's inequality, we have
\begin{equation}\label{Mty}
  \begin{split}
     \W_1((P_t^{\mu_1})^\ast\Gamma(\mu_2), \Gamma(\mu_2)) & \leq \left(\E[|X_t^1-X_t^2|^2+|Y_t^1-Y_t^2|^2]\right)^\frac{1}{2} \\
       & \leq \sqrt{\frac{2\kappa(\e^{(L_b+2)t}-1)}{L_b+2}}\W_1(\mu_1,\mu_2)=:G(t)\W_1(\mu_1,\mu_2),\quad t\geq0.
  \end{split}
\end{equation}
Set
$$\delta_1:=\left(\inf_{t>\frac{\log c_0}{\lambda_0}}\frac{\sqrt{\frac{2(%\e^{(1+L_b+K_1)t}-1
        \e^{(L_b+2)t}-1
       )}{L_b+2}}}{1-c_0 \e^{-\lambda_0 t}}\right)^{-2}.$$
By Theorem \ref{ABC}(i), when $\kappa<\delta_1$, we conclude that the solution of \eqref{Langevin} has a unique invariant probability measure $\mu^\ast \in \P_1(\R^{2d}).$

By a similar argument %to derive
for the derivation of
\eqref{w1erg}, we derive from {\bf(A2)} and It\^{o}'s formula that
$$\W_1(P_t^ \ast \eta, (P_t^ {\mu^\ast})^ \ast \eta)\leq 2c_0\sqrt{\kappa}\sqrt{\frac{\e^{(L_b+2+4\kappa)t}-1}{L_b+2+4\kappa}}\W_1(\eta,\mu^\ast), \quad \eta\in\scr P_1(\R^{2d}),\, t\geq0.$$
Set
$$H(t):=2c_0\sqrt{\kappa}\sqrt{\frac{\e^{(L_b+2+4\kappa)t}-1}{L_b+2+4\kappa}},\quad t\geq0,$$
and
$$\delta_2:=\sup\left\{\kappa>0:\ \inf_{t>\frac{\log c_0}{\lambda_0}}(H(t)+c_0\e^{-\lambda_0 t})<1\right\}. $$
Taking $\kappa<\delta_0:=\min\{\delta_1,\delta_2\}$, the proof is completed.
\end{proof}
 \section{Applications on $\alpha$-stable noise cases}
 \subsection{Non-degenerate noise case}
Let $W:=(W_t)_{t\geq 0}$ be an $n$-dimensional standard Brownian motion and $S:=(S_t)_{t\geq 0}$  be an $\frac{\alpha}{2}$-stable subordinator with $\alpha\in(1,2)$ independent of $W$. Then $(W_{S_t})_{t\geq 0}$ is a rotationally invariant $\alpha$-stable process (cf. \cite{Applebaum 2011}) with generator $-(-\Delta)^{\frac{\alpha}{2}}$.
We consider the following distribution dependent SDE on $\Rd$:
\begin{equation}\label{ddsde levy one}
  \d X_t= b(X_t,\L_{X_t}) \d t+\sigma(\L_{X_t}) \d W_{S_t},
\end{equation}
where $$b:\Rd \times \mathscr{P}(\mathbb{R}^{d})
 \rightarrow
   \Rd, \quad \sigma:\mathscr{P}(\mathbb{R}^{d})
    \rightarrow\mathbb{R}^{d}\otimes\R^{n}$$
    are measurable.
We make the following assumptions.
\begin{enumerate}
\item[{\bf (A3)}] $b$ is continuous on $\R^d\times\scr P_1(\R^d)$. There exist constants $K_1,K_2,\kappa, R>0, \delta>1$ such that for any $x,y\in \R^d,\, \mu_1, \mu_2 \in \P_1(\R^{d})$,
 \begin{align*}
&2\<b(x,\mu_1)-b(y,\mu_2), x-y\>\\
&\leq K_1|x-y|^21_{\{|x-y|^2\leq R\}}-K_2|x-y|^21_{\{|x-y|^2> R\}}+\kappa\W_1(\mu_1,\mu_2)^2,
\end{align*}
and
 \begin{equation*}
  \|\sigma(\mu_1)-\sigma(\mu_2)\|_{\mathrm{HS}}^2 \leq \kappa\W_1(\mu_1,\mu_2)^2, \quad  \delta^{-1}\leq \sigma\sigma^\ast    \leq \delta.
\end{equation*}

\end{enumerate}
Under {\bf (A3)},  \eqref{ddsde levy one} is well-posed in $\P_1(\Rd)$, as detailed in \cite[Theorem 1.1]{BJH arkiv}.
\begin{thm}\label{nonle} Assume {\bf(A3)}.
Then there exists a constant $\delta_0>0$ such that when $\kappa<\delta_0$, \eqref{ddsde levy one} has a unique invariant probability measure $\mu^\ast\in \P_1(\Rd)$ and there exists a constant $c,\lambda>0$ such that
\begin{equation}\label{speed levy one}
  \W_1(P_t^*\eta, \mu^\ast)\leq c\e^{-\lambda t}\W_1(\eta, \mu^\ast),\ \ \eta\in\scr P_1(\R^d),\, t\geq 0.
\end{equation}
\end{thm}
\begin{proof}
(1) For any $\mu\in\scr P_1(\R^d)$, consider
\begin{align}\label{Nonle1}\d X_t= b(X_t, \mu)\d t+  \sigma(\mu) \d W_{S_t}.
\end{align}
Thanks to \cite[Theorem 1.3]{LMJ 2020}, under {\bf(A3)}, for any $\mu\in\scr P_1(\R^d)$, \eqref{Nonle1} adimits a unique invariant probability measure $\Gamma(\mu)\in\scr P_1(\R^d)$ such that
\begin{align}
\label{cmtky}\W_1((P_t^{\mu})^\ast\nu,\Gamma(\mu))\leq c_0\e^{-\lambda_0 t}\W_1(\nu,\Gamma(\mu)),\ \ t\geq 0,\nu,\mu\in\scr P_1(\R^d)
\end{align}
for some constants $c_0\geq 1,\lambda_0>0$ independent of $\mu$.

So, by Theorem \ref{ABC}, to derive the existence and uniqueness of invariant probability measure, it remains to prove
$$\W_1((P_t^{\mu_1})^\ast\eta,(P_t^{\mu_2})^\ast\eta)\leq G(t)\W_1(\mu_1,\mu_2),\quad t\geq 0,\, %\eta\in\scr P_1(\R^d), \mu_i\in\scr P_1(\R^d), i=1,2
\eta,\mu_1,\mu_2\in\scr P_1(\R^d),$$
for some continuous function $G:[0,\infty)\to[0,\infty)$. To this end, we will use Lemma \ref{aul}. Consider the following SDEs
$$\d X_t^i=b(X_t^i,\mu_i)\d t+\sigma(\mu_i)\d W_{S_t}, \quad \mu_i\in \P_1(\Rd), i=1,2,$$
with $X_0^1=X_0^2$ and $\L_{X_0^1}=\eta.$
Using Lemma \ref{aul}, for any $t\geq0$, we arrive at
\begin{equation}\label{existence levy one}
  \begin{split}
     \E|X_t^1-X_t^2|
&\leq \sqrt{\int_0^t \e^{K_1(t-s)}\kappa\W_1(\mu_1,\mu_2)^2\,\d s}+\E\sqrt{\int_0^t\e^{K_1(t-s)}\kappa\W_1(\mu_1,\mu_2)^2\,\d S_s} \\
       & \leq \sqrt{\kappa}\W_1(\mu_1,\mu_2)\left(\sqrt{\int_0^t\e^{K_1(t-s)}\,\d s}+\E\sqrt{\int_0^t\e^{K_1(t-s)}\,\d S_s}\right)\\
       & \leq \W_1(\mu_1,\mu_2)\sqrt{\kappa}\e^{\frac{K_1}{2}t}
    \left(\sqrt{\frac{1-\e^{-K_1t}}{K_1}}+\E\sqrt{S_t}\right)=:G(t)\W_1(\mu_1,\mu_2).
  \end{split}
\end{equation}
Since $\alpha\in(1,2)$, we conclude that
\begin{equation}\label{mean-stable}
\E\sqrt{S_t}=t^{\frac{1}{\alpha}}\E\sqrt{S_1}< \infty, \quad t\geq 0.
\end{equation}
Set
$$\delta_1:=\left(\inf_{t>\frac{\log c_0}{\lambda_0}}\frac{\e^{\frac{K_1}{2}t}\
\left(\sqrt{\frac{1-\e^{-K_1t}}{K_1}}+t^{\frac{1}{\alpha}}\E\sqrt{S_1}\right)}{1-c_0 \e^{-\lambda_0 t}
}\right)^{-2}.$$
By Theorem \ref{ABC}(i), when $\kappa<\delta_1$, we conclude that \eqref{ddsde levy one} has a unique invariant probability measure $\mu^\ast \in \P_1(\R^{d}).$

 (2) Next we verify \eqref{speed levy one}. Firstly, it follows from \eqref{cmtky} that
\begin{align}\label{conve}\W_1((P_t^{\mu^\ast})^\ast\eta,\mu^\ast)\leq c_0\e^{-\lambda_0 t}\W_1(\eta,\mu^\ast),\quad t\geq 0,\eta\in\scr P_1(\R^d).
\end{align}
Next, we estimate $\W_1(P_{t}^\ast\eta,(P_{t}^{\mu^\ast})^\ast\eta)$.
Let $X_t^{i}, i=1,2$ be solutions of the following SDEs,
$$\d X_t^1=b(X_t^1,\mu^\ast)\d t+\sigma(\mu^\ast) \d W_{S_t},$$
$$\d X_t^2=b(X_t^2, P_t^\ast \eta)\d t+\sigma(P_t^\ast \eta) \d W_{S_t}$$
with $X_0^1=X_0^2$ and $\L_{X_0^1}=\eta.$
Again by Lemma \ref{aul} and \eqref{conve}, for any $t\geq 0$, we arrive at
\begin{equation*}
\begin{split}
&  \E|X_t^{1}-X_t^{2}|\\
&\leq \sqrt{\int_0^t\e^{K_1(t-s)}\kappa\W_1(P_s^\ast\eta,\mu^\ast)^2\,\d s}+\E\sqrt{\int_0^t\e^{K_1(t-s)}\kappa\W_1(P_s^\ast\eta,\mu^\ast)^2\,\d S_s} \\
&\leq \sqrt{\int_0^t\e^{K_1(t-s)}2\kappa c_0^2e^{-2 \lambda_0 s}\,\d s}\W_1(\eta,\mu^\ast)+\sqrt{\int_0^t\e^{K_1(t-s)}2\kappa \W_1((P_s^{\mu^\ast})^\ast\eta,P_s^\ast\eta)^2\,\d s}\\
&+\E\sqrt{\int_0^t\e^{K_1(t-s)}2\kappa c_0^2\e^{-2\lambda_0 s}\,\d S_s}\W_1(\eta,\mu^\ast)+ \E\sqrt{\int_0^t\e^{K_1(t-s)}2\kappa
\W_1((P_s^{\mu^\ast})^\ast\eta,P_s^\ast\eta)^2\,\d S_s}\\
&\leq \sqrt{2\kappa}c_0 \e^{\frac{K_1}{2}t} \left(\sqrt{\frac{1-\e^{-(K_1+2\lambda_0)t}}{K_1}}+\E\sqrt{S_t}\right)\W_1(\eta,\mu^\ast) \\
&+ \sqrt{2\kappa}\e^{\frac{K_1}{2}t}\sup_{s\in[0,t]}\W_1((P_s^{\mu^\ast})^\ast\eta,P_s^\ast\eta) \left(\sqrt{\frac{1-\e^{-K_1t}}{K_1}}+t^{\frac{1}{\alpha}}\E\sqrt{S_1}\right).
 \end{split}
\end{equation*}
This together with the fact %${\color{red}\E(\sqrt{S_t})}=t^{\frac{1}{\alpha}}\E\sqrt{S_1}$
\eqref{mean-stable}
implies that, for any $t\geq 0, \eta\in \P_1(\Rd),$
\begin{equation}\label{supty}
\begin{split}
&\sup_{s\in[0,t]}\W_1((P_s^{\mu^\ast})^\ast\eta,P_s^\ast\eta)\\
&\leq
\sqrt{2\kappa}c_0 \e^{\frac{K_1}{2}t}\left(\sqrt{\frac{1-\e^{-(K_1+2\lambda_0)t}}{K_1}}
+t^{\frac{1}{\alpha}}\E\sqrt{S_1}\right)\W_1(\eta,\mu^\ast)\\
&+\sqrt{2\kappa}\e^{\frac{K_1}{2}t}
\left(\sqrt{\frac{1-\e^{-K_1t}}{K_1}}+t^{\frac{1}{\alpha}}\E\sqrt{S_1}\right)  \sup_{s\in[0,t]}\W_1((P_s^{\mu^\ast})^\ast\eta,P_s^\ast\eta)\\
 &=:\sqrt{\kappa}h_1(t)\W_1(\eta,\mu^\ast)+\sqrt{\kappa}h_2(t)\sup_{s\in[0,t]}\W_1((P_s^{\mu^\ast})^\ast\eta,P_s^\ast\eta),
\end{split}
\end{equation}
where
$$ h_1(t):=\sqrt{2}c_0\e^{\frac{K_1}{2}t}\left(\sqrt{\frac{1-\e^{-(K_1+2\lambda_0)t}}{K_1}}+t^{\frac{1}{\alpha}}\E\sqrt{S_1}\right),\quad t\geq 0,$$
$$  h_2(t):=\sqrt{2}\e^{\frac{K_1}{2}t}\left(\sqrt{\frac{1-\e^{-K_1t}}{K_1}}+t^{\frac{1}{\alpha}}\E\sqrt{S_1}\right), \quad t\geq 0.$$
 When $\sqrt{\kappa}{h_2(t)}<1$, set
$$H(t):=\frac{\sqrt{\kappa} h_1(t)}{1-\sqrt{\kappa}h_2(t)}.$$
We derive from \eqref{supty} that
\begin{equation}\label{supte}
\begin{split}
\sup_{s\in[0,t]}\W_1((P_s^{\mu^\ast})^\ast\eta,P_s^\ast\eta)\leq H(t)\W_1(\eta,\mu^\ast).
 \end{split}
\end{equation}
Let
$$\delta_2=\sup\left\{\kappa>0:\ \inf_{t>\frac{\log c_0}{\lambda},\ \sqrt{\kappa}h_2(t)<1} (H(t)+c_0\e^{-\lambda_0t})<1\Big\}\right\}. $$
When $\kappa<\delta_0:=\min\{\delta_1,\delta_2\}$, we can find a constant $t_1>0$ satisfying
$H(t_1)+c_0\e^{-\lambda_0 t_1}<1$
such that $$\sup_{s\in[0,t_1]}\W_1(P_{s}^\ast\eta,(P_{s}^{\mu^\ast})^\ast\eta)\leq H(t_1)\W_1(\eta,\mu^\ast),  \quad \eta\in\scr P_1(\R^d).$$
Therefore, by Theorem \ref{ABC}(ii), we obtain \eqref{speed levy one}.
\end{proof}

 \subsection{Second order %system
system}
Let $\alpha\in(1,2)$.  Let $W:=(W_t)_{t\geq 0}$ be a $d$-dimensional standard Brownian motion and $S:=(S_t)_{t\geq 0}$ be an $\frac{\alpha}{2}$-stable subordinator independent of $W$. In this part, we consider the following second order system driven by % $\alpha$-stable processes
the $\alpha$-stable process
$(W_{S_t})_{t\geq 0}$, i.e., %for $\alpha\in(1,2)$, {\color{red}where $W:=(W_t)_{t\geq 0}$ is a $d$-dimensional standard Brownian motion and $S:=(S_t)_{t\geq 0}$ is an $\frac{\alpha}{2}$-stable subordinator with $\alpha\in(1,2)$ independent of $W$.}
 \begin{equation}\label{Langevin levy}
\begin{cases}
\d X_t=Y_t\d t, \\
\d Y_t=-\gamma Y_t\d t+ b(X_t, \L_{(X_t, Y_t)}) \d t +\sigma (\L_{(X_t, Y_t)}) \d W_{S_t},
\end{cases}
\end{equation}
where  $\gamma>0$ is a constant, and
$$b:\Rd \times \mathscr{P}(\mathbb{R}^{2d})\rightarrow\Rd, \quad \sigma:\mathscr{P}(\mathbb{R}^{2d})\rightarrow\mathbb{R} %, \quad \gamma>0.
    $$
    are measurable.

Under assumptions {\bf (A2)}, it is well known that for each $\F_0$-measurable %and
$\R^{2d}$-valued random variables $(X_0,Y_0)$ with $\L_{(X_0,Y_0)}\in\P_1(\R^{2d})$, \eqref{Langevin levy} admits a unique solution in $\P_1(\R^{2d})$ (cf. \cite{BJH arkiv}).
\begin{thm}\label{dle} Assume {\bf(A2)} for
$$L_b^2\gamma^{-2}<\frac{3}{4}K_1.$$
Then there exists a constant $\delta_0>0$ such that when $\kappa<\delta_0$, \eqref{Langevin levy} has a unique invariant probability measure $\mu^\ast\in\scr P_1(\R^{2d})$ and there exist constants $c,\lambda>0$ such that
\begin{equation*}
  \W_1(P_t^*\eta, \mu^\ast)\leq ce^{-\lambda t}\W_1(\eta, \mu^\ast),\quad\eta\in\scr P_1(\R^{2d}),\, t\geq 0.
\end{equation*}
\end{thm}
\begin{proof}
(1) For any $\mu\in\scr P_1(\R^{2d})$, consider the decoupled SDE:
 \begin{equation}\label{decsd}
\begin{cases}
\d X_t^{\mu}=Y_t^{\mu}\d t, \\
\d Y_t^{\mu}=-\gamma Y_t^{\mu}\d t+b (X_t^{\mu},\mu) \d t+  \sigma(\mu)\d W_{S_t}.
\end{cases}
\end{equation}
As in the proof of Theorem \ref{dBM}, let $(\bar{X}_t^\mu,\bar{Y}_t^\mu)=\sigma(\mu)^{-1}(X_t^\mu,Y_t^\mu)$. Then it holds
 \begin{equation}\label{transfor}
\begin{cases}
\d \bar{X}_t^{\mu}=\bar{Y}_t^{\mu}\d t, \\
\d \bar{Y}_t^{\mu}=-\gamma \bar{Y}_t^{\mu}\d t+\bar{b}(\bar{X}_t^{\mu}, \mu)\d t+  \d W_{S_t},
\end{cases}
\end{equation}
where
$$\bar{b}(x,\mu)=\sigma(\mu)^{-1}b(\sigma(\mu)x,\mu).$$
Let $(\bar{P}_t^\mu)^\ast\eta$ be the distribution of $(\bar{X}_t^\mu,\bar{Y}_t^\mu)$ with initial distribution $\eta\in\scr P_1(\R^{2d})$.
As claimed in the proof of Theorem \ref{dBM}, {\bf (A2)} implies \eqref{bar-b} and \eqref{pdi}.
By \cite[Theorem 1.1]{LWZ 2023}, \eqref{transfor} has a unique invariant probability measure $\bar{\Phi}(\mu)\in\scr P_1(\R^{2d})$ such that
\begin{equation*}
  \W_1((\bar{P}_t^\mu)^\ast\bar \eta, \bar{\Phi}(\mu))\leq c_0\e^{-\lambda_0 t}\W_1(\bar \eta, \bar{\Phi}(\mu)),\quad \bar \eta\in\scr P_1(\R^{2d}),\, t\geq 0,
\end{equation*}
for some %constant
constants
 $c_0\geq 1,\lambda_0>0$ independent of $\mu$.

Define
$$\Phi(\mu)(\cdot)=\bar{\Phi}(\mu)(\sigma(\mu)~\cdot).$$
Then $\Phi(\mu)\in\scr P_1(\R^{2d})$ is the unique invariant probability measure of \eqref{decsd} and similar to \eqref{gap}, it holds
\begin{align*}
  \W_1((P_t^\mu)^\ast\eta, \Phi(\mu))\leq c_0\e^{-\lambda_0 t}\W_1(\eta, \Phi(\mu)),\ \ \eta\in\scr P_1(\R^{2d}),\, t\geq 0.
\end{align*}
By {\bf(A2)}, for any $x,y,\bar{x},\bar{y}\in \Rd, \mu_1,\mu_2 \in \P_1(\R^{2d})$, we have
\begin{align*}
&2\<y-\bar{y},x-\bar{x}\>+2\<-\gamma (y-\bar{y})+b(x,\mu_1)-b(\bar{x},\mu_2),y-\bar{y}\>\\
&\leq (|x-\bar{x}|^2+|y-\bar{y}|^2)+(L_b+1)(|x-\bar{x}|^2+|y-\bar{y}|^2)+ \kappa \W_1(\mu_1, \mu_2)^2\\
&\leq (L_b+2)(|x-\bar{x}|^2+|y-\bar{y}|^2)+ \kappa \W_1(\mu_1, \mu_2)^2.
\end{align*}
The remaining is repeating the proof of Theorem \ref{nonle} by replacing $K_1$ with $L_b+2$ starting from \eqref{cmtky}. So, we omit it.
\end{proof}

\section{Appendix}
In this section, we shall derive a crucial lemma in the $\alpha$-stable noise case by means of the time change argument in \cite{WFY 2014, ZXC 2013}.
Consider
\begin{align}\label{sdrei}
\d X_t^i=b(X_t^i,\mu_t^i)\d t+\sigma^i_t\d W_{S_t},\ \ i=1,2,
\end{align}
where
 $b:\R^d\times\scr P(\R^d)\to\R^d$, $\sigma^i:[0,\infty)\to\R^d\otimes\R^n$ are measurable and locally bounded, and $\mu_\cdot^i\in C([0,\infty),\scr P_1(\R^d))$. As previously defined, $ (W_{S_t})_{t \geq 0} $ is a rotationally invariant $n$-dimensional $\alpha$-stable process with $\alpha \in (1, 2)$.
%为了和前面统一，$\R^n$都改成了$\R^d$,后面同理。
\begin{enumerate}
\item[{\bf (H)}] For any $\mu\in\scr P_1(\R^d)$, $b(\cdot,\mu)$ is continuous.
 There exist constants $
K,\kappa\geq 0$ such that for any $x,y\in \R^d,\, \mu_1, \mu_2 \in \P_1(\R^{d})$,
\begin{align*}
&2\<b(x,\mu_1)-b(y,\mu_2), x-y\>\leq K|x-y|^2+\kappa\W_1(\mu_1,\mu_2)^2.
\end{align*}
%and
% \begin{equation*}\begin{split}
% & \|\sigma(\mu_1)-\sigma(\mu_2)\|_{HS}^2 \leq K\W_1(\mu_1,\mu_2)^2.
% \end{split}\end{equation*}

\end{enumerate}
\begin{lem}\label{aul}
Assume {\bf (H)}. Then it holds
\begin{align}\label{ctyka}
\E|X_t^{1}-X_t^{2}|
\nonumber&\leq e^{\frac{K}{2}t}\E|X_0^{1}-X_0^{2}|\\
&+\sqrt{\int_0^te^{K(t-s)}\kappa\W_1(\mu_s^1,\mu_s^2)^2\,\d s}+\E\sqrt{\int_0^te^{K(t-s)}\|\sigma^1_s-\sigma^2_s\|_{\mathrm{HS}}^2\,\d S_s}.
\end{align}
\end{lem}
\begin{proof}  Let $\ell$ be a path of $S$ and consider
\begin{align}\label{lsde1}
\d X_t^{i,\ell}=b(X_t^{i,\ell},\mu_t^i)\d t+\sigma_t^i\d W_{\ell_t},%\ \ i=1,2,
\end{align}
with $X_0^{i,\ell}=X_0^i$, $i=1,2$. We shall regularize $\ell$ %to
and
transform \eqref{lsde1} to %%%%%%%%%%%%%%%%5%be
an SDE driven by standard Brownian motions. For this purpose, letting $\epsilon \in (0,1)$, we define
\begin{align}\label{cty01}\ell_t^\epsilon:=\frac{1}{\epsilon}\int_{t}^{t+\epsilon} \ell_s\, \d s+\epsilon t=\int_{0}^{1} \ell_{\epsilon s+t} \, \d s+\epsilon t,\quad t\geq0.
\end{align}
%for $\epsilon \in (0,1).$
Then, $t\mapsto \ell^\epsilon_t$ is strictly increasing, absolutely continuous and for every $t\geq0$, $\ell^\epsilon_t \downarrow \ell_t$ as $\epsilon \downarrow 0.$
%Let $\ell^\epsilon$ be defined in \eqref{cty01} and consider
Consider
\begin{align*}
\d X_t^{i,\epsilon}=b(X_t^{i,\epsilon},\mu_t^i)\d t+\sigma_t^i\d W_{\ell^{\epsilon}_t},%\ \ i=1,2
\end{align*}
with $X_0^{i,\epsilon}=X_0^i$, $i=1,2.$

(1) We first assume that
\begin{enumerate}
\item[(i)] For each $k=1,2$, $\sigma^k$ is %$\sigma^i, i=1,2$, are
piecewise constant %, which
and can be expressed as follows
$$%\sigma_{\cdot}
\sigma_{\cdot}^k=\sum_{i=1}^{\infty}\mathbbm{1}_{[t_{i-1},t_i)}(\cdot)%\sigma_{t_i}
\sigma_{t_i}^k,$$
where $(t_i)_{i\in \mathbb{N}}\subseteq\mathbb{R}$ is a sequence with $t_0=0$ and $t_i\uparrow\infty$ as $i \uparrow \infty$;
\item[(ii)] There exists a constant $L>0$ such that
\begin{align*}
&|b(x,\mu_1)-b(y,\mu_1)|\leq L|x-y|, \quad  x,y\in \Rd,\, \mu_1\in \scr P_1(\R^d).
\end{align*}
\end{enumerate}
Then, by \cite[Lemma 3.3]{WFY 2014}, we conclude that for any $t\geq 0$,
\begin{align}\label{ckt}
\lim_{\epsilon\to 0^+} X_t^{i,\epsilon}=X_t^{i,\ell}, \quad i=1,2.
\end{align}
By It\^{o}'s formula, we have
\begin{align*}
\d |X_t^{1,\epsilon}-X_t^{2,\epsilon}|^2&\leq K|X_t^{1,\epsilon}-X_t^{2,\epsilon}|^2\d t+\kappa\W_1(\mu_t^1,\mu_t^2)^2\d t+\|\sigma^1_t-\sigma^2_t\|_{\mathrm{HS}}^2\d \ell^{\epsilon}_t\\
&+2\<[\sigma^1_t-\sigma^2_t]\d W_{\ell^{\epsilon}_t},X_t^{1,\epsilon}-X_t^{2,\epsilon}\> , \quad t\geq 0.
\end{align*}
Gronwall's inequality implies that
\begin{align*}
&\E(|X_t^{1,\epsilon}-X_t^{2,\epsilon}|^2|\F_0)\\
&\leq e^{Kt}|X_0^{1}-X_0^{2}|^2+\int_0^te^{K(t-s)}\kappa\W_1(\mu_s^1,\mu_s^2)^2\,\d s+\int_0^te^{K(t-s)}\|\sigma^1_s-\sigma^2_s\|_{\mathrm{HS}}^2\,\d \ell^{\epsilon}_s ,\quad t\geq 0.
\end{align*}
As $\epsilon\rightarrow0^+$, %It
it follows from Fatou's lemma and \eqref{ckt} that
\begin{align*}
&\E(|X_t^{1,\ell}-X_t^{2,\ell}|^2|\F_0)\\
&\leq e^{Kt}|X_0^{1}-X_0^{2}|^2+\int_0^te^{K(t-s)}\kappa\W_1(\mu_s^1,\mu_s^2)^2\,\d s+\int_0^te^{K(t-s)}\|\sigma^1_s-\sigma^2_s\|_{\mathrm{HS}}^2\,\d \ell_s, \quad t\geq 0.
\end{align*}
So, Jensen's inequality yields
\begin{align}\label{x1-2}
\nonumber&\E|X_t^{1,\ell}-X_t^{2,\ell}|\\
&\leq e^{\frac{K}{2}t}\E|X_0^{1}-X_0^{2}|+\sqrt{\int_0^te^{K(t-s)}\kappa\W_1(\mu_s^1,\mu_s^2)^2\,\d s}\\
\nonumber&+\sqrt{\int_0^te^{K(t-s)}\|\sigma^1_s-\sigma^2_s\|_{\mathrm{HS}}^2\,\d \ell_s},\quad t\geq 0.
\end{align}

(2) We assume (ii).
Fix $T>0$. We can choose $\{\sigma^{1,m}\}_{m\geq 1}$ and $\{\sigma^{2,m}\}_{m\geq 1}$ satisfying that for any $m\geq 1$, $\sigma^{1,m}$ and $\sigma^{2,m}$ are piecewise constant and as $m\to\infty$, $\sigma^{i,m}\to \sigma^{i}$ in $L^2([0,T];\d \ell)$  for each $i=1,2$.
Consider
\begin{align*}
\d X_t^{i,m,\ell}=b(X_t^{i,m,\ell},\mu_t^i)\d t+\sigma_t^{i,m}\d W_{\ell_t},\ \ i=1,2.
\end{align*}
Then, by \cite{WFY 2014}, we conclude $\lim_{m\to\infty}\E|X_t^{i,m,\ell}-X_t^{i,\ell}|=0$.
By (1), \eqref{x1-2} holds for $(X_t^{i,m,\ell}, \sigma^{i,m})$ replacing $(X_t^{i,\ell}, \sigma^{i})$. Letting $m\to\infty$, we derive \eqref{x1-2}.

(3) In general, we adopt the Yosida approximation used in step (c) of \cite[Proof of Theorem 2.1]{WFY 2014}. Let
    $$
        \tilde{b}(x,\mu):=b(x,\mu)
        -\frac12 Kx,\ \ x\in\R^{d},\,\mu\in\P_1(\Rd).
    $$
    Then $\tilde{b}$ is also continuous and
    \begin{align*}
  2\<\tilde{b}(x,\mu_1)-
        \tilde{b}(y,\mu_2),
        x-y\>\leq \kappa\W_1(\mu_1,\mu_2)^2,\ \ m\geq 1,\, x,y\in\R^d,\, \mu_1,\mu_2\in\scr P_1(\R^d).
    \end{align*}
    Let $\mathrm{id}$ be the identity map on $\R^d$. For any $m\geq 1$, %let
   consider the Yosida approximation of $\tilde{b}$, i.e.,
    $$
        \tilde{b}^{(m)}(x,\mu)
        :=m\left[
        \left(\operatorname{id}
        -\frac{1}{m}\tilde{b}(\cdot,\mu)\right)^{-1}(x)-x
        \right],\quad x\in\R^d,\,\mu\in\scr P_1(\R^d).
    $$
 %After applying the Yosida approximation, the resulting operator
Then, $|\tilde{b}^{(m)}|\leq|\tilde{b}|$, $\tilde{b}^{(m)}$ is globally Lipschitz continuous in $x$ and converges pointwise to
    $\tilde{b}$ as $m$ tends to $\infty$, allowing us to utilize the conclusion of (2).
      For any $x,y\in\R^d,\, \mu_1,\mu_2\in\scr P_1(\R^d)$, let $$\bar{x}=\left(\operatorname{id}
        -\frac{1}{m}\tilde{b}(\cdot,\mu_1)\right)^{-1}(x),\quad \bar{y}=\left(\operatorname{id}
        -\frac{1}{m}\tilde{b}(\cdot,\mu_2)\right)^{-1}(y).$$
        Then we have $$x=\bar{x}
        -\frac{1}{m}\tilde{b}(\bar{x},\mu_1),\quad y=\bar{y}
        -\frac{1}{m}\tilde{b}(\bar{y},\mu_2),$$
        which implies that for any $x,y\in \Rd, \,\mu_1,\mu_2 \in\P_1(\Rd),$
        \begin{align}\label{jg4dv70}
        \nonumber&2\<\tilde{b}^{(m)}(x,\mu_1)-
        \tilde{b}^{(m)}(y,\mu_2),
        x-y\>=\<m(\bar{x}-x)-m(\bar{y}-y),x-y\>\\
        &=\<\tilde{b}(\bar{x},\mu_1)-\tilde{b}(\bar{y},\mu_2), \bar{x}-\bar{y}-\frac{1}{m}(\tilde{b}(\bar{x},\mu_1)-\tilde{b}(\bar{y},\mu_2))\>\\
       \nonumber &\leq \kappa %\W_1(\mu_1,\mu_2)^2%%%)
       \W_1(\mu_1,\mu_2)^2,\quad m\geq 1.
    \end{align}
    Set
    $$b^{(m)}(x,\mu):=\tilde{b}^{(m)}(x,\mu)+\frac12 Kx, \quad x\in\R^d, \,\mu\in\scr P_1(\R^d).$$
Then, by \eqref{jg4dv70}, for any $m\geq 1$, there exists a constant $L_m>0$ such that (ii) holds for $(b^{(m)},L_m)$ replacing $(b,L)$ and
$b^{(m)}$ satisfies
    \begin{align*}\label{jg4dv7}
        \nonumber&2\<b^{(m)}(x,\mu_1)-
        b^{(m)}(y,\mu_2),
        x-y\>\\
        &\leq K|x-y|^2+\kappa\W_1(\mu_1,\mu_2)^2,\quad m\geq 1,\, x,y\in\R^d,\, \mu_1,\mu_2\in\scr P_1(\R^d).
    \end{align*}

For any $m\geq 1$ and for each $i=1,2$, let $X_t^{i,(m)}$ solve   %Let $X_t^{i,(m),\ell},i=1,2$ solve
\eqref{lsde1} with $b^{(m)}$ replacing $b$ and $X_0^{i,(m),\ell}=X_0^i$. We have $\mathbb{P}$-a.s. $\lim_{m\to\infty}X_t^{i,(m),\ell}=X_t^{i,\ell}$. By (2), \eqref{x1-2} holds for $X_t^{i,(m),\ell}$ replacing $X_t^{i,\ell}$. Letting $m\to\infty$ and applying Fatou's lemma, we obtain \eqref{x1-2}.
Since $W$ and $S$ are independent, we have $\E\{\E(|X_t^{1,\ell}-X_t^{2,\ell}|)_{\ell=S}\}=\E(|X_t^{1}-X_t^{2}|)$, which together with \eqref{x1-2} completes the proof.
\end{proof}
\begin{rem} Let $\nu_Z$ be the L\'{e}vy measure associated to $W_{S_t}$. To prove \eqref{ctyka}, another alternative procedure is to adopt It\^{o}'s formula for $|X_t^1-X_t^2|^2$ and by doing so, the %a
term
\begin{align*}&\int_{\R^d}\{|X_t^1-X_t^2+(\sigma^1_t-\sigma^2_t)z|^2-|X_t^1-X_t^2|^2-2\<X_t^1-X_t^2,(\sigma^1_t-\sigma^2_t)z\>1_{\{|z|\leq 1\}}\}\,\nu_Z(\d z)
\end{align*}
appears. It is not difficult to see that
\begin{align*}
&\int_{\R^d}\{|X_t^1-X_t^2+(\sigma^1_t-\sigma^2_t)z|^2-|X_t^1-X_t^2|^2- 2\<X_t^1-X_t^2,(\sigma^1_t-\sigma^2_t)z\>1_{\{|z|\leq 1\}}\}\,\nu_Z(\d z)\\
&\leq\int_{|z|>1}2\<X_t^1-X_t^2,(\sigma^1_t-\sigma^2_t)z\>\,\nu_Z(\d z)+\int_{|z|>1}|(\sigma^1_t-\sigma^2_t)z|^2\,\nu_Z(\d z)\\
&+\int_{|z|\leq 1}|(\sigma^1_t-\sigma^2_t)z|^2\,\nu_Z(\d z),
\end{align*}
%which involves in $\int_{|z|>1}|z|^2%\nu(\d z){\color{red}\nu_Z(\d z)}.$
where $\int_{|z|>1}|z|^2\,\nu_Z(\d z)$ involves in.
However, it is known that $\int_{|z|>1}|z|^2\,\nu_Z(\d z)=\infty$. This is the reason why we adopt the method of time-change instead.
\end{rem}

\end{document}